\documentclass[amssymb,amsfonts,12pt,verbatim,righttag,oneside]{amsart}

\usepackage{comment}
\usepackage{bbm}
\usepackage{amsmath}
\usepackage{amsfonts}
\usepackage{amssymb}
\usepackage{enumerate}
\usepackage{graphicx}
\usepackage{verbatim}
\usepackage{graphics,color}
\usepackage[colorlinks,citecolor=blue,pagebackref,urlcolor=black,hypertexnames=false]{hyperref}

\hypersetup{
    colorlinks=true,
    linkcolor=blue,
    filecolor=magenta,
    urlcolor=cyan,
    pdftitle={Overleaf Example},
    pdfpagemode=FullScreen,
    }

\numberwithin{equation}{section} 
\theoremstyle{plain}

\definecolor{trp}{rgb}{1,1,1}

\definecolor{red}{rgb}{1,0,.2}
\newcommand*{\clrred}[1]{{\color{red} #1}}

\definecolor{blue}{rgb}{0,0,1}
\newcommand*{\clrblue}[1]{{\color{blue} #1}}
\newcommand{\cblue}[1]{\clrblue{ #1}}

\definecolor{rgrey}{rgb}{.8,0.4,.4}  
\definecolor{grey}{rgb}{.13,.13,.13}  
\definecolor{lightgray}{rgb}{0.83, 0.83, 0.83}

\definecolor{green}{rgb}{0.0,0.4,0.2}

\newcommand{\fm}{\ensuremath{\clrred{}}}
\newcommand{\fmu}{\ensuremath{\,}}
\newcommand{\fma}{\ensuremath{\,}}

\newcommand{\indicator}{\ind}
\newcommand{\Leb}{{\small \textsc{Leb}}}
\newcommand{\Lebfrac}{{\textsc{Leb}}}

\newcommand{\dimH}{\ensuremath{\dim_{\rm H}}}

\newcommand*{\di}{\, \mathrm{d} }

\newcommand*{\ev}{\ensuremath{\mathbb{E}}}
\newcommand*{\pr}{\ensuremath{\mathbb{P}}}

\newcommand*{\ind}{\mathbf{1}}
\newcommand*{\ds}{\ensuremath{\displaystyle }}
\newcommand*{\eps}{\ensuremath{\varepsilon }}

\newcommand*{\Z}{\ensuremath{\mathcal{Z}}}

\newcommand*{\vect}[1]{\ensuremath{\underline{#1}}}
\newcommand*{\T}{\ensuremath{T}}
\newcommand*{\HH}{\ensuremath{T}}
\newcommand*{\Heta}{\ensuremath{H}}
\newcommand*{\J}{\ensuremath{\textsc{J}}}
\newcommand*{\xm}{\ensuremath{x_{\max}}}

\newtheorem{theorem}{Theorem}
\theoremstyle{plain}

\newtheorem{claim}[theorem]{Claim}

\newtheorem{corollary}[theorem]{Corollary}

\newtheorem{definition}[theorem]{Definition}

\newtheorem{fact}[theorem]{Fact}
\newtheorem{lemma}[theorem]{Lemma}

\newtheorem{proposition}[theorem]{Proposition}
\newtheorem*{proposition*}{Proposition}
\newtheorem{remark}[theorem]{Remark}

\newtheorem{mainlemma}[theorem]{Main Lemma}

\renewcommand*{\vect}[1]{\ensuremath{\mathbf{ #1}}}

\graphicspath{{./figuresl/}}

\usepackage{float}

\begin{document}

\baselineskip 18pt


\title{The interior of randomly perturbed self-similar sets on the line}

\author{Michel Dekking}
\address{Michel Dekking,\fm CWI, Amsterdam and Delft Institute of Applied Mathematics,
 Technical University of Delft,  The Netherlands
}
\email{F.M.Dekking@ewi.tudelft.nl}

\author{K\'{a}roly Simon}
\address{K\'{a}roly Simon,
Department of Stochastics, Budapest University of
Technology and Economics,
MTA-BME Stochastics Research Group, Budapest University of Technology and
Economics, Műegyetem rkp. 3, H-1111 Budapest, Hungary,
Alfr\'ed R\'enyi Institute, Hungary
}\email{simonk@math.bme.hu}

\author{Bal\'{a}zs Sz\'{e}kely}
\address{Bal\'{a}zs Sz\'{e}kely, Department of Stochastics, Budapest University of
Technology and Economics, H-1529 B.O.box 91, Hungary
}\email{szbalazs@math.bme.hu}

\author{N\'{o}ra Szekeres}
\address{N\'{o}ra Szekeres, Department of Stochastics, Budapest University of
Technology and Economics, H-1529 B.O.box 91, Hungary
}\email{szenora269@gmail.com }

 \thanks{
\indent The research of Simon was supported by NKFI Foundation
K142169  and the NWO-OTKA common project. The research of
Sz\'{e}kely was supported by HSN Lab}

\begin{abstract} Can we find a self-similar set on the line with positive Lebesgue measure and empty interior?  Currently, we do not have the answer for this question for deterministic self-similar sets. In this paper we answer  this question negatively for random self-similar sets   which are defined with the construction introduced in the paper
Jordan, Pollicott and Simon (Commun. Math. Phys., 2007). For the same type of random self-similar sets we prove the Palis-Takens conjecture which asserts that
at least typically the algebraic difference of dynamically defined Cantor sets is either large in the sense that it contains an interval or small in the sense that it is a set of zero Lebesgue measure.
\end{abstract}

\maketitle

\tableofcontents

\section{Introduction}
In this paper we consider only Random self-similar Iterated Function Systems (RIFS) which are defined on the line and
which can be obtained as a small random perturbation of a deterministic self-similar Iterated Function System (IFS) on the line.
First we give a short description of our results for the expert, and then we provide a more detailed introduction. We do not write "self-similar"
in the abbreviation since all iterated function systems considered in this paper are self-similar (random or deterministic).

\subsection{Informal description of the  main result for experts}
Using the construction introduced by Jordan, Pollicott and Simon
\cite[p. 521]{jordan2007hausdorff}, we define self-similar  Random Iterated Function Sytems (RIFS) $\mathcal{F}$ on the line as follows:
 We start with a self similar IFS  $\mathcal{S}$ on the line and we add a small random additive error to every map
in every step of the iterative construction of the attractor (see Definition \ref{def:RIFS} for the \fm precise definition of RIFSs).
The \fm scaling parts of the \fm similarities of the deterministic IFS $\mathcal{S}$ are left unchanged. So, the similarity dimension $s(\mathcal{F})$
of the RIFS $\mathcal{F}$ is the same \fm as the  similarity dimension of the deterministic self-similar IFS $\mathcal{S}$. Our main result is that
\begin{equation}
    \label{y98}
    s(\mathcal{F})>1    \Longrightarrow  \mathrm{int}(C_{\mathcal{F}})\ne \emptyset, \quad \text{almost surely},
    \end{equation}
where $C_{\mathcal{F}}$ is the attractor of the RIFS $\mathcal{F}$.
This implies that whenever $C_1, \ C_2$ are two independent copies of the attractor of the RIFS $\mathcal{F}$
then the algebraic difference set $C_2-C_1:=\left\{\fmu c_2-c_1:c_1\in C_1, c_2\in C_2 \right\}$ satisfies
\begin{equation}
\label{y44}
s(\mathcal{F})>\frac12
\Longrightarrow
\text{int}(C_2-C_1)\ne \emptyset.
\end{equation}

\subsection{A gentle introduction}
A deterministic self-similar Iterated Function System (IFS) on $\mathbb{R}$ is a finite list of contracting similarities  of $\mathbb{R}$:
\begin{equation}
\label{z02}
\mathcal{S}:=\fm\fm\left\{ S_i: S_i(x)=r_ix+t_i,\; x\in \mathbb{R} \right\}_{i=1}^{L},
\end{equation}
with contractions $r_i\in(-1,1)\setminus\left\{ 0 \right\}$ and translations\fm\footnote{We do not assume that the $t_i$ are distinct} $t_i\in\mathbb{R}$ for all $i\in[L]:=\left\{ 1,\dots,L \right\}$.
The attractor $C_{\mathcal{S}}$ of the IFS $\mathcal{S}$ is what we are left with after infinitely many iterations of the system.
More formally, it is easy to see that we can find a non-degenerate compact interval $I$ such that $S_i(I)\subset I$ holds for all $i\in[L]$.
For all $\mathbf{i}\in[L]^n$ we consider the $n$-fold iterate
\begin{equation}
\label{z01}
S_{\mathbf{i}}:=S_{i_1}\circ\cdots\circ S_{i_n}
\end{equation}
and form the corresponding $n$-cylinder interval $I_{\mathbf{i}}:=S_{\mathbf{i}}(I)$. Then the union of all $n$-cylinders
$\big\{\bigcup\limits_{\mathbf{i}\in[L]^n} I_{\mathbf{i}}  \big\}_{n=1}^{\infty  }$ is a nested sequence of non-empty compact sets. The attractor is their intersection:
\begin{equation}
\label{y99}
C_{\mathcal{S}}:=\bigcap\limits_{n=1 }^{\infty   }
\bigcup\limits_{\mathbf{i}\in[L]^n} I_{\mathbf{i}}.
\end{equation}
In this paper, we consider random IFSs (RIFS) on $\mathbb{R}$, which are small (translational) perturbations
of a self-similar IFS of the form \eqref{z02}.

\medskip

Informally, the attractor of \fmu an RIFS is obtained by a formula similar to \eqref{y99} with the following difference:
Instead of the  deterministic $n$-cylinder intervals $I_{\mathbf{i}}$ in \eqref{y99},
we work with the random intervals  \fm $\widehat{I}_{\mathbf{i}}=f_{i_1}\circ\cdots\circ f_{i_n}(I)$, where  the random mappings $f_{i_k}$ are small translational perturbations of $S_{i_k}$. Namely, $f_{i_k}=S_{i_k}+\fm Y _{i_k} $ for small random translations $ \fm Y _{i_k}$.

The precise description of the distribution of these random translations is given in Definition \ref{def:RIFS}.
The attractor $C_{\mathcal{F}}$ of the RIFS $\mathcal{F}:=\left\{f_i \right\}_{i=1}^{L}$ is defined by a formula analogous to \eqref{y99}\fm: we just replace $I_{\mathbf{i}}$ with $\widehat{I}_{\mathbf{i}}$ in \eqref{y99}.

Jordan, Pollicott, and Simon \cite{jordan2007hausdorff} studied this kind of
 RIFSs in the more general self-affine case. As an immediate consequence of the results in \cite{jordan2007hausdorff}, we get that the Hausdorff dimension $\dim_{\rm H}
 C_{\mathcal{F}}$ is the minimum of $1$ and the similarity dimension $s_{\mathcal{F}}$ (solution of \eqref{eq:S_dim}) almost surely. Moreover, if
 $s_{\mathcal{F}}>1$ then the Lebesgue measure of
 $C_{\mathcal{F}}$ is positive almost surely.

\subsection{Motivation: the interior of the difference of random Cantor sets }

In 1987, Palis and Takens \cite{PT} studying the dynamical behavior of diffeomorfisms presented a conjecture about the size of the algebraic difference of two Cantor sets. Informally,
if the size of the Cantor sets is large (see equation \eqref{P})
then the difference contains an interval. More precisely, if $C_1$ and $C_2$ are two Cantor sets then the algebraic difference
\begin{equation*}
C_2-C_1=\{y-x: x\in C_1,y\in C_2\}
\end{equation*}
contains an interval if
\begin{equation}\label{P}
\dim_{\rm H}C_1+\dim_{\rm H}C_2>1,
\end{equation}
where $\dim_{\rm H}$ denotes the Hausdorff dimension.

In 2001, De Moreira and Yoccoz (\cite{MY}) proved the conjecture for generic  dynamically generated \emph{non-linear} Cantor sets.
The conjecture has not been proven for generic linear Cantor sets.

In 1990, Per Larsson put the problem into a probabilistic context in
\cite{Larssonthesis}, (see also \cite{Larsson}). He considered a
very special family of two parameter random Cantor sets $C_{a,b}$ and proved
the conjecture for a certain subset of $a$'s and $b$'s. Although the main idea
of Larsson's argument is brilliant, unfortunately the proof contains
significant gaps and incorrect reasonings. In 2011, three out of the four
authors of the present paper gave a precise proof for Larsson's family in
\cite{dekking2011algebraic}. We briefly recall the Larsson family from
\cite{dekking2011algebraic}: let
\begin{equation*}\label{2}
a>\frac{1}{4}\quad  \mbox{ and }  \quad 3a+2b<1.
\end{equation*}
\noindent Since
$$\dim_{\rm H} C_{a,b}=-\frac{\log 2}{\log a},$$
the first condition is equivalent to $\dim_{\rm H}C_{a,b}>1/2$,
which is equivalent to equation (\ref{P}).

Larsson's construction is as follows (see also
Figure~\ref{fig:Cantor}): first remove the interval
$$\left[\frac{1}{2}-\frac{a}{2},\frac{1}{2}+\frac{a}{2}\right]$$ from
the middle of $[0,1]$, then the \fm length $b$ parts from both the beginning
and the end of the unit interval. Next, put  intervals of length $a$
according to a uniform distribution in the remaining two gaps
$\left[b,\frac{1}{2}-\frac{a}{2}\right]$ and
$\left[\frac{1}{2}+\frac{a}{2},1-b\right]$. These two randomly
chosen intervals of length $a$ are called the level one intervals of
the random Cantor set $C_{a,b}$.
\begin{figure}[!t]
    \centering
    \includegraphics[width=130mm]{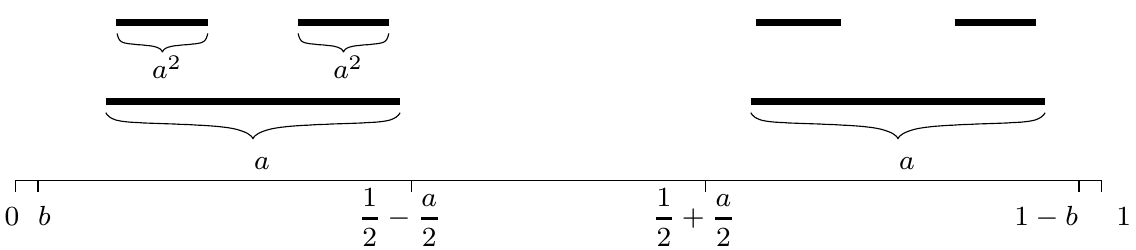}
   \caption{The construction of the Cantor set $C_{a,b}$.
    The figure shows the level 1 and level 2 cylinder intervals $C_{a,b}^{1}$ and $C_{a,b}^{2}$.}
    \label{fig:Cantor}
\end{figure}

We write $C_{a,b}^{1}$ for their union. In both of the two level one
intervals we repeat the same construction independently of each
other and of the previous step. In this way we obtain four disjoint
intervals of length $a^2$.  We emphasize that, because of
independence, the relative positions of these second level intervals
in the first level ones are in general completely different.
Similarly, we construct the $2^n$  level $n$ intervals of length
$a^n$. We call their union $C_{a,b}^{n}$. Then Larsson's random
Cantor set is defined by
$$
C_{a,b}:=\bigcap _{n=1}^\infty C_{a,b}^{n}.
$$

As a corollary of the main result of this paper, we prove that the conjecture by Palis and Takens
holds for a very broad class of \emph{random} linear Cantor sets,
including the Larsson family.

The following result is a generalization of the result in \cite{dekking2011algebraic}:

\begin{theorem}\label{y48}
 Let $\mathcal{F}$ be \fm an RIFS (see Definition \ref{def:RIFS}) with similarity dimension larger than $\frac12$.
Let $C _1$ and $C _2$ be two independent copies of the attractor $C
_{\mathcal{F}}$. Then\\[-.6cm]
$$ C _2-C _1 \mbox{ contains an interval a.s.}$$
\end{theorem}

 The proof is presented in Section \ref{y35}.

It is important to note that in our setting the Hausdorff dimension equals the \fm (the minimum of 1 and) the similarity dimension given by the unique solution of equation (\ref{eq:S_dim}).

We remark that if the Hausdorff dimension of $C_{\mathcal{F}}$ is
smaller than $\frac{1}{2}$ then the set $C _2-C _1$ has Hausdorff
dimension less than $1$  so it cannot possibly contain any
intervals.

The essential part of the proof of this theorem is completely different of that of the main result in \cite{dekking2011algebraic}.
The proof in \cite{dekking2011algebraic} was tailored for the Larsson's family, and does not have the potential for generalizations.
However, a combination of  the ideas of the proof in
\cite{dekking2011algebraic} with the method introduced in Rams, Simon \cite{Rams-Simon} and also invoking
an observation from Peres, Shmerkin \cite{Peres-Shmerkin} makes it possible to prove a much more general result with a shorter proof.
Namely, both in  \cite{dekking2011algebraic} and the present paper we have to verify that the associated multi-type branching processes are uniformly supercritical, where uniformity is meant in the type of the ancestor.
In both papers this is stated as the Main Lemma and their  proofs follow the same path.
However, the step where using the Main Lemma one  proves the existence of intervals in the arithmetic difference of the random Cantor sets, is where we use the method introduced in \cite{Rams-Simon} and this makes our present proof much more efficient.

We remark here that if \emph{also} the scalings are random, then  the problem becomes easier (\cite{oral-Shmerkin}), and can be treated by a method introduced by Hochman and Shmerkin (\cite{Hochman-Shmerkin}).

\section{RIFS}

\subsection{The formal description of our random Cantor set}\label{304}

First, we give the formal definition of the Random Iterated Functions System (RIFS) $\mathcal{F}$, whose attractor
$C_\mathcal{F}$ is the random Cantor set which is the object of our investigation in this paper.
It is convenient to identify the collection of all finite words over the alphabet $[L]=\{1,2,\dots,L\}$ with the nodes  of the $L$-ary tree $\mathcal{T}$.
\fmu The empty word is identified with the root of $\mathcal{T}$, and  denoted as  $\emptyset$.
For any $n\ge 1$ the level $n$ sets $\mathcal{L}_n$  of $\mathcal{T}$ are defined by
\begin{equation*}
 \mathcal{L}_0=\emptyset, \quad \mathcal{L}_n=\{i_1\dots i_n: i_j\in [L], 1\le j\le n\}.
\end{equation*}

\begin{definition}[RIFS]\label{def:RIFS}
Let
\begin{equation}\label{306}
  \mathcal{F}=\left\{ f_i:f_i(x)=r_ix+D_i\right\}_{i=1}^{L}.
\end{equation}
The contraction ratios
$r_1,\dots ,r_L\in (-1,1)\setminus{\{0\}}$ are deterministic.
 About the random translations  $(D_1,\dots,D_L)$,  of the functions in $\mathcal{F}$ in  (\ref{306}), we assume the following:

\bigskip

\begin{enumerate}[{\bf (a)}]
\item
{\emph{ $(D_1,\dots,D_L)$ is an $L$-dimensional random variable such that for
    any $i=1,\dots,L$, the random variable $D_i$ is  absolutely continuous w.r.t. the Lebesgue
    measure, with a density function $\varphi _i$ which is strictly positive, bounded and continuous on $(t_i-\theta_i,t_i+\theta_i)$} and \fm $\varphi_i$ is zero outside $(t_i-\theta_i,t_i+\theta_i)$}, \fm where the $t_i$ and $\theta_i>0$ are real numbers.
\item To define the random translations of the iterates of this system we introduce
$$\ds \fm \left\{D^{(\vect{i})}= \fm \left(D^{(\vect{i})}_1,\dots,\fm D^{(\vect{i})}_L \right) \right\}_{\vect{i}\in\mathcal{T}}$$
as a set of i.i.d.~random \fm vectors having the same distribution as that of $(D_1,\dots,D_L)$.

 The iterates  $f_{\mathbf{i}}$ for $\mathbf{i}\in  \mathcal{L}_n$  are defined as follows:
\begin{eqnarray*}\label{307}
\nonumber  f_\mathbf{i}(x) &=& f_{i_1}\circ\dots\circ f_{i_n}(x)\\
   &=& r_{i_1}\left(r_{i_2}\left(\dots \left(r_{i_{n-1}}(r_{ i_n}x+D^{(i_1\dots
   i_{n-1})}_{i_n})+D^{(i_1\dots i_{n-2})}_{i_{n-1}}\right)\dots
   \right)+D^{(i_1)}_{i_2}\right)+D^{\emptyset}_{i_1}\\
\nonumber   &=& r_{\mathbf{i}}x+T_{\mathbf{i}}
\end{eqnarray*}
where  $r_{\vect{i}}=r_{i_1}\cdots r_{i_n}$ and
\begin{equation}\label{eq:311}
T_{\vect{i}}= D^{\emptyset}_{i_1}+r_{i_1}D^{(i_1)}_{i_2}+r_{i_1}r_{i_2}D^{(i_1i_2)}_{i_3}+\ldots
+ r_{i_1}\cdots r_{ i_{n-1}}D^{(i_1\dots i_{n-1})}_{i_n}.
\end{equation}

\end{enumerate}

\end{definition}

Using that the random mappings $f_i$, $i\in[L]$ are contractions and the supports of  the $D_i$ are bounded, we immediately obtain that
there  exists  a deterministic interval $[\alpha,\beta]$ such that
\begin{equation}
\label{y84}
f_i([\alpha,\beta])\subset [\alpha,\beta],\ \text{ for all } i\in[L].
\end{equation}

\fm We call $[\alpha,\beta]$ a \emph{supporting interval} for $\mathcal{F}$.

\medskip

The attractor $C _{\mathcal{F}}$ of the RIFS $\mathcal{F}$  is defined by
\begin{equation}
\label{y83}
C _{\mathcal{F}}=\bigcap\limits _{n=1}^{\infty   }
\bigcup\limits_{|\mathbf{i}|=n}  f_{\mathbf{i}}([\alpha,\beta]).
\end{equation}

\begin{figure}[ht!]
    \includegraphics[width=140mm]{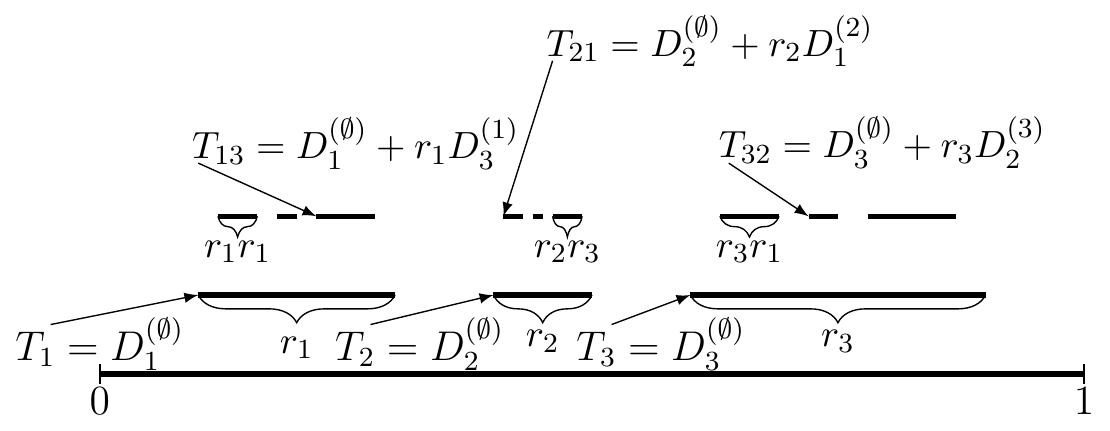}
    \caption{Level 1 and 2 cylinder intervals of our Cantor set when $L=3$ and  $[\alpha,\beta]=[0,1]$. The randomly chosen
    left endpoints $T_i$ and some \fma of the $T_{ij}$, $i,j=1,2,3$ are  indicated.\fma}
        \label{fig:Cantor_set}
    \end{figure}

\subsection{The ambient probability space}\label{Probspace}
 The sample space is $\Omega :=[\mathbb{R}^L]^{\mathcal{T}}$.
The corresponding $\sigma$-algebra $\mathcal{B}$ is the generated Borel $\sigma$-algebra. The probability measure of our Cantor set is
\begin{equation*}
\pr=\prod_{\vect{i}\in \mathcal{T}}d\left(D^{(\vect{i})}\right),
\end{equation*}
where $d(X)$ denotes the probability distribution of a random vector $X$.
Then the ambient probability space is \fm $(\Omega ,\mathcal{B},\pr)$.
A realization $\pmb{\omega}\in\Omega $ is a labelled tree, $\pmb{\omega}=\left\{ \omega _{\mathbf{i}} \right\}_{\mathbf{i}\in \mathcal{T}}$,
where $\omega _{\mathbf{i}}$ is defined for an $\mathbf{i}=i_1\dots  i_n$ as follows:
$$
\omega _{i_1}=D _{i_1}^{\emptyset  }(\pmb{\omega}) \text{ if $n=1$ and }
\text{ and }
\omega _{i_1\dots  i_n}=D _{ i_n}^{(i_1\dots  i_{n-1}) }(\pmb{\omega}) \fm \text{ if } n>1.
$$
 According to \eqref{eq:311}, \fm $T_{\mathbf{i}}(\pmb{\omega}) = \omega _{i_1}+\sum _{k=2}^{n} r_{i_1\dots  i_{k-1}}\omega_{i_1\dots  i_k}.$

 The dimension theory of the RIFS described above is well understood.
 The following Fact is a  \fm direct  consequence of the results in
 \cite{jordan2007hausdorff}.
\begin{fact}[Dimension of an RIFS]\label{prop:S_dim}
Let $\mathcal{F}$ be an RIFS of size $L$ and let $s(\mathcal{F})$
denote the solution \fm to the equation
\begin{equation}\label{eq:S_dim}
\sum_{i=1}^L|r_i|^{s(\mathcal{F})}=1.
\end{equation}
We say that $s(\mathcal{F})$ is the similarity dimension of $\mathcal{F}$.
Then we have for almost all realizations:
\begin{equation*}
\dimH C_{\mathcal{F}}= \overline{\mathrm{dim}}_BC_{\mathcal{F}}=
\underline{\mathrm{dim}}_BC_{\mathcal{F}}=
\min\left\{ 1, s(\mathcal{F}) \right\}
\end{equation*}
Moreover, if $s(\mathcal{F})>1$ then for almost all realizations:
\begin{equation}
\label{z61}
\Leb\left( C_{\mathcal{F}} \right)>0,
\end{equation}
where \Leb$(\cdot)$ is the $1$-dimensional Lebesgue measure.
\end{fact}

With these definitions we can state our Main Theorem:
\begin{theorem}[Main Theorem]\label{thm:main}
Let $C_{\mathcal{F}}$ be the attractor of  an RIFS $\mathcal{F}$ with $s(\mathcal{F})>1 $. Then
\begin{equation}
\label{z60}
\mathrm{int}(C_{\mathcal{F}})\ne \emptyset.
\end{equation}
\end{theorem}
The proof is presented in Section \ref{y35}.

\subsection{The $n$-th order RIFS}\label{z59}

To introduce \fm the notion of an $n$-th order RIFS  we  consider the size $n$ subtrees of $\mathcal{T}$ defined by
\begin{equation*}
 \mathcal{T}_n(\emptyset)=\bigcup_{j=0,\dots, n-1}\mathcal{L}_j,
\quad \mathcal{T}_n(\vect{k})=\bigcup_{j=0,\dots, n-1}\{\vect{i}\in\mathcal{L}_{\ell+j}: i_1\dots i_\ell=\vect{k}\},
\end{equation*}
for each \fm $\vect{k}=k_1\dots k_\ell$.
\begin{definition}\label{def:nthorder}
Let $\mathcal{F}=\{f_i\}_{i=1}^L$ be an RIFS.
The  $n$-th order of $\mathcal{F}$, written as $\mathcal{F}^n$, is defined for $n=1,2,\dots$ by
\begin{equation*}
 \mathcal{F}^n=\{f_{\vect{i}}\}_{\vect{i}\in \mathcal{L}_n},
\end{equation*}
\end{definition}

Actually $\mathcal{F}^n$ is itself an RIFS.

\begin{lemma}\label{lem:nthorder}
Let $\mathcal{F}$ be an RIFS  and let  $\mathcal{F}^n$ \fm be its $n$-th order, for some $n\ge 1$.
Then  $\mathcal{F}^n$ is an RIFS.
\end{lemma}

\begin{proof}

Recall that the elements of $\mathcal{F}^n$ indeed have the form $f_{\vect{i}}(x)=r_{\vect{i}}x+T_{\vect{i}}$, and (see \eqref{eq:311}) that the
  random translations $T_{\vect{i}}$ satisfy \fm
\begin{equation}\label{eq:3111}
T_{\vect{i}}= D^{\emptyset}_{i_1}+r_{i_1}D^{(i_1)}_{i_2}+r_{i_1}r_{i_2}D^{(i_1i_2)}_{i_3}+\ldots
+ r_{i_1}\cdots r_{i_1\dots i_{n-1}}D^{(i_1\dots i_{n-1})}_{i_n}.
\end{equation}

Note first that  $(T_{\vect{i}}:\vect{i}\in \mathcal{L}_n)$   is an $L^n$-dimensional random vector such that for any $\vect{i}$, the random variable $T_{\vect{i}}$ is bounded, and absolutely continuous w.r.t. the Lebesgue measure, supported \fm on  an interval  with strictly positive and continuous density on its interior.

\bigskip

Let us write $\overrightarrow{D}=(D_1,\dots,D_L)$, and \fm $\overrightarrow{D}^{\vect{j}}=(D_1^{(\vect{j})},\dots,D_L^{(\vect{j})})$ for $\vect{j}\in\mathcal{T}$, where the $D_i^{(\vect{j})}$ are i.i.d.~for $i\in{L}$.
All the random variables in Equation (\ref{eq:3111}) are coming from i.i.d. random vectors $\overrightarrow{D}^{\vect{j}}$ with $\vect{j}\in \mathcal{T}_n(\emptyset)$.

For a set of nodes $\mathcal{N}$, let us write $\overrightarrow{D}^{\mathcal{N}}$ for the random vector
with elements $\overrightarrow{D}^{\vect{i}}$ with $\vect{i}\in \mathcal{N}$.

We see from Equation (\ref{eq:3111}) that for $\vect{i}\in  \mathcal{L}_n$ the translation of $f_{\vect{i}}$ is completely determined by $\overrightarrow{D}^{ \mathcal{T}_n(\emptyset)}$. More generally, for all $k\ge 1$ the translations of  $\vect{i}\in  \mathcal{L}_{kn}$ are completely determined by $\overrightarrow{D}^{ \mathcal{T}_n(\vect{j})}$, where $\vect{j}$ is the unique ancestor of $\vect{i}$ in $\mathcal{L}_{(k-1)n}$. Note that $\overrightarrow{D}^{ \mathcal{T}_n(\vect{j})}$ has the same distribution as $\overrightarrow{D}^{ \mathcal{T}_n(\emptyset)}$ for all $\vect{j}\in \mathcal{T}$. Since the $\overrightarrow{D}^{ \mathcal{T}_n(\vect{j})}$ are independent for all $k\ge 1$ and $\vect{j}\in \mathcal{L}_{(k-1)n}$, by disjointness of the size $n$ subtrees rooted at the levels that are multiples of $n$, it follows that we have the required independent iteration scheme for $\mathcal{F}^n$.
 \end{proof}

 \bigskip

 At the cost of lowering the similarity dimension with an arbitrary small amount,
we can assume that all the scalings in an RIFS are positive: see \cite[Lemma 2.10]{farkas2019dimension} and a remark in the proof of Proposition 6 in \cite{Peres-Shmerkin}\fm.
For completeness we combine \fm the results of these references in the following proposition and its proof.

\begin{proposition}\label{prop:Positive}  Let $\mathcal{F}=\left\{f_i:f_i(x)=r_ix+D_i\right\}_{i=1}^{L}$ be an RIFS with attractor $C_{\mathcal{F}}$.\, Then there exists for every $\varepsilon>0$ an RIFS $\widetilde{\mathcal{F}}$ with  $C_{\widetilde{\mathcal{F}}}\subset  C_{\mathcal{F}}$, such that all the contraction ratios   of the similarities in $\widetilde{\mathcal{F}}$ are positive, and $s(\widetilde{\mathcal{F}})>s(\mathcal{F})-\varepsilon $.
\end{proposition}

\begin{proof} In case all $r_i$ are positive, there is nothing to prove. Otherwise we may assume w.l.o.g.~that $r_1<0$. For a natural number $n$, which will be chosen conveniently large later in the proof, we define for all $\mathbf{i}=i_1\dots i_n\in \mathcal{L}_n$
$$\widetilde{r_{\mathbf{i}}}=r_1r_{\mathbf{i}}\quad {\rm if\;} r_{\mathbf{i}}<0, \qquad
  \widetilde{r_{\mathbf{i}}}=   r_{\mathbf{i}}\quad {\rm if\;} r_{\mathbf{i}}>0.$$
  Here the $r_{\mathbf{i}}$ are the contraction ratios of the  $f_{\vect{i}}(x)=r_{\vect{i}}x+T_{\vect{i}}$ mappings from $\mathcal{F}^n$.
Since  $s:=s(\mathcal{F})$ is the solution to $\sum_1^L |r_i|^s=1$, we must have $\sum_1^L |r_i|^{s-\varepsilon}>1$ for all $\varepsilon>0$.
So
\begin{eqnarray*}\label{eq:dimtilde}
\sum_{\mathbf{i}\in \mathcal{L}_n} |\widetilde{r_{\mathbf{i}}}|^{s-\varepsilon} &=&
    \sum_{\widetilde{r_{\mathbf{i}}}<0} |\widetilde{r_{\mathbf{i}}}|^{s-\varepsilon}  + \sum_{\widetilde{r_{\mathbf{i}}}>0} |\widetilde{r_{\mathbf{i}}}|^{s-\varepsilon} =
   \sum_{r_{\mathbf{i}}<0}\fm |r_1|^{s-\varepsilon} |r_{\mathbf{i}}|^{s-\varepsilon}  + \sum_{r_{\mathbf{i}}>0} |r_{\mathbf{i}}|^{s-\varepsilon}\\
      &\ge& \fm |r_1|^{s-\varepsilon}  \sum_{\mathbf{i}\in \mathcal{L}_n} |r_{\mathbf{i}}|^{s-\varepsilon} =  \fm |r_1|^{s-\varepsilon} \bigl(\,\sum_1^L |r_i|^{s-\varepsilon}\,\bigr)^n > 1,
\end{eqnarray*}
were we have taken  $n$ such that $\bigl(\,\sum_1^L |r_i|^{s-\varepsilon}\,\bigr)^n>\fm |r_1|^{-s+\varepsilon}$.
Conclusion: if we choose $\widetilde{\mathcal{F}}$ with functions $\widetilde{f_{\mathbf{i}}}$ defined by  $\widetilde{f_{\mathbf{i}}}(x)=\widetilde{r_{\mathbf{i}}}(x)+T_{\vect{i}}$, then $C_{\widetilde{\mathcal{F}}}\subset  C_{\mathcal{F}}$ and $s(\widetilde{\mathcal{F}})>s(\mathcal{F})-\varepsilon$.
\end{proof}

\section{\fm It is enough to consider homogeneous systems}

 We call an RIFS \emph{homogeneous} if all contraction ratios are the same. Using a simple combination of \cite[Lemma 2.8]{farkas2019dimension} and \cite[Proposition 6]{Peres-Shmerkin} it appears that any RIFS can be well-approximated by a homogeneous RIFS. This is Proposition \ref{y89}.

\medskip

Given an RIFS $\mathcal{F}$ of the form \eqref{306}.
Let $U\subset \mathcal{L}_{n}$, $\# U\geq 2 $ for an $n\geq 1$. We define
\begin{equation}
\label{y90}
\mathcal{F}_U=\left\{ f_{\mathbf{i}} \right\}_{\mathbf{i}\in U}=\left\{ f_{\mathbf{i}}: f_{\mathbf{i}}(x)=   r_{\mathbf{i}}x+T_{\mathbf{i}}\right\}_{\mathbf{i}\in U}.
\end{equation}
Here the random variables $T_{\mathbf{i}}$ are defined in \eqref{eq:311}.
\fm According to Lemma \ref{lem:nthorder}, $\mathcal{F}^n$ is an RIFS. Then $\mathcal{F}_U$ is
also an RIFS since $\mathcal{F}_U$ is a subsystem of $\mathcal{F}^n$.
For all realizations, the random attractor  $C_{\mathcal{F}_U}$ of $\mathcal{F}_U$ is a subset of the random attractor
$\mathcal{C}_{\mathcal{F}}$ of $\mathcal{F}$, for the same realization.

\begin{proposition}\label{y89}\fm
  Let $ \mathcal{F}=\left\{f_i \fma:f_i(x)=r_ix+D_i\right\}_{i=1}^{L}$ be an RIFS as in Definition \ref{def:RIFS}.

  Then there exists for every $\varepsilon>0$ a number $n$,  a set $U\subset \mathcal{L}_{n}$, and an $a\in (0,1)$ such that the RIFS $\mathcal{F}_U$ has the form
  \begin{equation}
\label{y86}
\mathcal{F}_U=\left\{ f_{\mathbf{i}}: f_{\mathbf{i}}(x)=   ax+T_{\mathbf{i}}  \right\}_{\mathbf{i}\in U},
\end{equation}
 and  satisfies   $C_{\mathcal{F_U}}\subset  C_{\mathcal{F}}$,   $s(\mathcal{F_U})>s(\mathcal{F})-\varepsilon $.
\end{proposition}

 For the convenience of the reader, we give a detailed proof of this result in the Appendix.

\section{The Main Lemma}

A homogeneous system $\mathcal{H}$ has the form
\begin{equation}
\label{z98}
\mathcal{H}:=\left\{ H_i:H_i(x) =ax+D_i\right\}_{i=1}^{L},\quad a\in(0,1).
\end{equation}

 Here we are motivated by  Proposition \ref{prop:Positive}: with an arbitrary small loss in similarity dimension we may assume that $a\in(0,1)$ instead of $a\in(-1,1)\setminus \{0\}$.

\medskip

\fm It is convenient to introduce a slightly unusual notation. \fm The \emph{support of a function} $f:X\to \mathbb{R}$, for an arbitrary set $X$, is the set-theoretical support. That is
  \begin{equation}
  \label{z08}
  \mathrm{supp}(f):=\left\{ x\in X: f(x)\ne 0 \right\}.
  \end{equation}
  This is slightly unusual since  $\mathrm{supp}(f)$ most commonly means  the closure of the set in \eqref{z08}.

\begin{proposition}\label{y46}
Let $\mathcal{H}$ be a homogeneous RIFS  and let $C_1$ and $C_2$ be two independent copies of the attractor of the RIFS.
Then  the algebraic difference $C_2-C_1$ is  the attractor of   a homogeneous RIFS $\mathcal{H}^\circleddash$ with similarity dimension $s(\mathcal{H}^\circleddash)=2s(\mathcal{H})$.
\end{proposition}

\begin{proof}
  Let the homogeneous RIFS $\mathcal{H}$ be given by
  \begin{equation}
    \label{y43}
    \mathcal{H}:=\left\{ H_i(x):H_i(x) =ax+D_i\right\}_{i=1}^{L },\quad
    a\in(0,1),
    \end{equation}
where $\overrightarrow{D}=(D_1,\dots,D_L)$ is an $L$ dimensional random \fm vector  such that for every $i\in[L]$
the random variable $D_i$ is  absolutely continuous w.r.t.~the Lebesgue measure, with a density $\varphi _i$ which is bounded,
continuous with $\mathrm{supp}(\varphi_i)=(t_i-\theta_i,t_i+\theta_i)$.

    For every $i\in[L]$ let $\widehat{D}_i\stackrel{d}{=}D_i$ and $\widetilde{D}_i\stackrel{d}{=}D_i$. Moreover, we require that
    $\widehat{D}_i$ and $\widetilde{D}_i$ are independent. We define
\begin{equation}
\label{y45}
D_{i,j}:=\widetilde{D}_i-\widehat{D}_j,\qquad
(i,j)\in[L]\times [L].
\end{equation}
Then $D_{i,j}$ is absolutely continuous w.r.t.~the Lebesgue mesure, with a density function $\varphi _{i,j}$ which is continuous, bounded with

\qquad   $\mathrm{supp}(\varphi _{i,j})=(t_i-t_j-\theta _i-\theta _j,t_i-t_j+\theta _i+\theta _j)$.\\
   That is, $\varphi _{i,j}$ satisfies all the requirements we set for the density function in Part (a) of Definition \ref{def:RIFS}.
   Using that, \fm we can define the homogeneous RIFS which consists \fm of a number of $L^2$ functions
$\mathcal{H}^\circleddash:=  \left\{ ax+D_{i,j} \right\}_{i,j=1}^{L}$. It is straightforward to check that the attractor
    $\Lambda^\circleddash$ of $\mathcal{H}^\circleddash$ is the algebraic difference of two independent copies of the attractor of $\mathcal{H}$.
\end{proof}

In the rest of the paper, we  always assume that the following assumptions hold:

\begin{enumerate}
[{\bf {A}1}]
\item $\mathcal{H}$ is a homogeneous RIFS $\left\{ H_i:H_i(x) =ax+D_i\right\}_{i=1}^{L}$ with supporting interval $I=[\alpha ,\beta ]$, and $ s(\mathcal{H})>1$.
\item $\mathcal{H}$ is a random perturbation of the deterministic IFS $\mathcal{S}:=\left\{ ax+t_i \right\}_{i=1}^{L}$.
\end{enumerate}

Here {\bf A2} means the following:  for all $i\in[L]$ there exist an interval $[-\theta_i,\theta_i]$ and an absolutely continuous  random variable $Y_i$
    whose probability  density function $\widetilde{f}_i$ is continuous, bounded and has support $\mathrm{supp}(\widetilde{f}_i)= (-\theta_i,\theta_i),$ such that
\begin{equation}\label{z97}
D_i\stackrel{d}{=}t_i+Y_i.
\end{equation}

\medskip

 The self-similarity property of an RIFS is expressed by the position of a point $x\in J_i:=H_i(I)$ relative to the endpoints of $J_i$, for $i\in [L]$.
This corresponds to the position of a point that we call $\Phi_i(x)$ relative to the endpoints of the supporting interval $I=[\alpha,\beta]$. This leads to the following definition.

\medskip

\begin{definition}\label{z86}
\begin{enumerate}
[{\bf (a)}]
    \item \fm For an $i\in [L] $, we write $J_i=H_i(I)=[A_i,B_i]$,    and we define the random variable
\begin{equation}\label{z85}
\Phi_i(x ):=\frac{x-A_i}{a}+\alpha \quad\hbox{if } x\in J_i, \qquad \Phi_i(x ):=\Theta \quad\hbox{if } x\not\in J_i,
\end{equation}
 where the symbol $\Theta$ represents that  $x\not\in J_i$.
    \item For an $x\in \mathrm{int}(I)$ let $\phi_i(x,\cdot)$ be the density function of $\Phi _i(x)$. Then an easy calculation yields that
\begin{equation}
\label{z82}
\phi_i (x,y)=a\widetilde{f}_i\left( x-t_i-ay \right),
\end{equation}
where $\widetilde{f}_i$ is \fm\fm the density of the random variables \fm $Y_i$ defined by (\ref{z97}).\\
The following function plays a crucial role in our argument
\begin{equation}
\label{z81}
m_I(x,y):=\sum _{i=1}^{L } \phi _i(x,y),\quad
(x,y)\in I^2.
\end{equation}
\end{enumerate}
\end{definition}

\medskip

\begin{mainlemma}\label{z80}
    There exists a  $T(0)\subset  I$ (the {\rm  pre-type space}) which is composed of a finite number of disjoint open intervals and there exists a real number \fm$\eps_{{\rm\scriptscriptstyle{MAIN}}}>0$ such that for every \fm$\eps \in (0,\eps_{{\rm\scriptscriptstyle{MAIN}}})$, the so-called {\rm type space} \fm
    \begin{equation}\label{eq:typespace} T(\eps ):=T(0)\setminus B(\partial T(0),\eps ),\end{equation}
     where $B(E,\eps):=\bigcup\{(x-\eps,x+\eps):x\in E\}$ for an $E\subset \mathbb{R}$, satisfies
    \begin{enumerate}
    \item The compact set $T(\eps )$ consists of as many intervals as $T(0)$.
    \item Let $m^{\varepsilon}:= m^{\varepsilon}_1:=m_I\cdot\fm  \indicator_{T(\varepsilon)\times T(\varepsilon)}$ and for $n\geq 1$ let
    \begin{equation}
    \label{y80}
    m^{\varepsilon}_{n+1}(x,y):= \int\limits_{T(\varepsilon)} m^{\varepsilon}_{n}(x,z)\cdot m^{\varepsilon}_{1}(z,y)\, \mathrm dz.
    \end{equation}
        Then there is an index $N_0$ for which the function $m_{N_0}^{\eps }$ is uniformly positive and bounded on  $T(\eps )\times T(\eps)$.
    \item The Perron-Frobenius eigenvalue of the operator \fm
$$F^{\eps }:\ h(x)\mapsto \int\limits_{T} m_1^{\eps}(x ,y)\cdot h(y)\, \di y, \quad x\in T(\eps )$$
acting on  $L^2(T(\varepsilon ))$  is larger than $1$.
\item The corresponding eigenfunction $f_{\varepsilon }(x)$ is continuous on $T(\eps )$.
\item {$H_{\mathbf{i}}(T(0))\subset T(0)$ for every $\mathbf{i}\in\mathcal{T}$.}
    \end{enumerate}
    \end{mainlemma}
    The proof is given in Section \ref{y81}.

    \medskip

    The   contents of the following theorem form the essential part of the proof of our main result Theorem \ref{thm:main}.

    \begin{theorem}\label{y79}
      Let $\mathcal{H}$ be a  homogeneous RIFS  with $s(\mathcal{H})>1$.  Then $\text{int}(C_{\mathcal{H}})\ne \emptyset $ almost surely.
    \end{theorem}

 \noindent  Theorem \ref{y79} will be proved in Section \ref{sec:Th12}, as a consequence of the Main Lemma \ref{z80}.

\section{Proof of Theorem \ref{y48}  and Theorem \ref{thm:main} assuming Theorem \ref{y79} }\label{y35}

 We first prove the Main Theorem:

\begin{proof}[Proof of Theorem \ref{thm:main}]
  Given an RIFS $ \mathcal{F}$  with similarity dimension $s(\mathcal{F})>1$. \fm
   Then according to Proposition \ref{y89}, we can find a homogeneous RIFS   $\mathcal{H}:=\mathcal{F}_U$, with $s(\mathcal{H})>1$.
 It follows from Theorem \ref{y79} that the attractor \fm $C_{\mathcal{H}}$ of the homogeneous RIFS $\mathcal{H}$ contains an interval almost surely.
Then this interval is also contained in the attractor $C_{\mathcal{F}}$ of the RIFS $\mathcal{F}$.
\end{proof}
\begin{proof}[Proof of Theorem \ref{y48}]
Given an RIFS $ \mathcal{F}$ with similarity dimension larger than $\frac{1}{2}$.
Let $\varepsilon >0$ so that $\varepsilon<s(\mathcal{F})-\frac12$.
Using Proposition \ref{y89} there exists a homogeneous   \fmu RIFS $\mathcal{H}:=\mathcal{F}_U$,
with $C_\mathcal{H}\subseteq C_\mathcal{F}$ and  $s(\mathcal{H})>s(\mathcal{F})-\varepsilon>\frac12$.
Let $C_1$, $C_2$ be two independent copies of the attractor of $\mathcal{F}$. Then we also have two independent copies $C^1_\mathcal{H}\subseteq C_1$ and
 $C^2_\mathcal{H}\subseteq C_2$ where  $s(\mathcal{H})>\frac12$.
According to \fm Proposition \ref{y46} we can find another homogeneous RIFS $\mathcal{H}^\circleddash$ whose attractor is $C^1_\mathcal{H}-C^2_\mathcal{H}$, and $s(\mathcal{H}^\circleddash)=2s(\mathcal{H})>1$.
So by  Theorem \ref{thm:main}  $C^2_\mathcal{H}-C^1_\mathcal{H}$ contains an interval almost surely. But then also $C_2-C_1\supseteqq C^1_\mathcal{H}-C^2_\mathcal{H}$ contains an interval almost surely.
\end{proof}

\section{The multi type branching process $\mathcal{Z}$}

\fm
On the probability space $\Omega$ we define a multi type branching process
$\mathcal{Z}=(\mathcal{Z}_n)_{n=0}^\infty$.   For a fixed a $0<\varepsilon <\varepsilon _{{\rm\scriptscriptstyle{MAIN}}}$  the type space from Equation \eqref{eq:typespace} is denoted by
$T:=T(\varepsilon)$. Actually, the \fm set of types is $T\cup \left\{ \Theta \right\}$.
By definition $\Theta$ is an absorbing state. That is, all descendants of $\Theta$ \fm are  equal to $\Theta$.

Let $\mathcal{Z}_0=\{x\}$, where $x\in T$, and for an $i\in[L]$ let $Z_{i}:=\Phi_i(x)$. Then
\begin{equation*}
\mathcal{Z}_1:=\left\{Z_{i}:i=1,\dots,L\right\}.
\end{equation*}
Note that although we speak of $\Theta$ as a type, it is \emph{not} an element of $T\subset [\alpha ,\beta ]$.\fm

\medskip

 To define the $\mathcal{Z}_n$, we need some preparations.
We follow the definition of an RIFS given in Definition \ref{def:RIFS}, but now from the viewpoint of the $Y$-vectors, instead of the $D$-vectors.
So we are given
 $$\left\{Y^{(\vect{i})}= \left(Y^{(\vect{i})}_1,\dots, Y^{(\vect{i})}_L \right) \right\}_{\vect{i}\in\mathcal{T}}$$
as a set of i.i.d.~random  vectors having the same distribution as that of $(Y_1,\dots,Y_L)$.


For an \fm $\mathbf{i}=i_1\dots  i_n$ we define
\begin{equation}
t_{\mathbf{i}}:=\sum\limits _{k=1}^{ n} a^{k-1}\cdot t_{i_k},\
Y_{\mathbf{i}}:= Y _{i_1}^{\emptyset }+\sum \limits _{k=2}^n a^{k-1}\cdot Y _{i_k}^{(i_1i_2\dots i_{k-1})}
 \text{ and }\theta_{\mathbf{i}}:=\sum _{k=1}^{n}a^{k-1}\cdot \theta_{i_k}.
\end{equation}
Clearly, the iterates from (\ref{307}) take the following form for a homogeneous RIFS
\begin{equation}
\label{y82}
H_{\mathbf{i}}(x)=a^n x+T_{\mathbf{i}},\ \text {where } T_{\mathbf{i}}:=t_{\mathbf{i}}+Y_{\mathbf{i}}.
\end{equation}
It follows from our assumption on the density  $\widetilde{f}_i$ of $Y_i$ \fm
that  (identifying somewhat carelessly the support of an absolutely continuous random variable with the support of its density function)
\begin{equation}
\label{z88}
\mathrm{supp}(Y_{\mathbf{i}})= (-\theta_{\mathbf{i}},\theta_{\mathbf{i}}).
\end{equation}

 \fm Let $I=[\alpha,\beta]$ be the supporting interval of $\mathcal{H}$.  We define the level-$n$ (random) cylinder intervals
$$
J_{\mathbf{i}}:=H_{\mathbf{i}}(I)= [a^n \alpha +T_{\mathbf{i}},a^n \beta + T_{\mathbf{i}}]\subset
[a^n \alpha +t_{\mathbf{i}}-\theta_{\mathbf{i}},a^n \beta  +t_{\mathbf{i}}+\theta _{\mathbf{i}}].
$$
The collection of all of these random level $n$ intervals is denoted by $\mathfrak{I}_n$. Note that these are intervals of length $(\beta -\alpha )a^n$.
The endpoints of the random interval  $J_{\mathbf{i}}\in \mathfrak{I}_n$ are $A_{\mathbf{i}}\fm=a^n \alpha +T_{\mathbf{i}}$ and $B_{\mathbf{i}}=A_{\mathbf{i}}+(\beta -\alpha )a^n$. That is, by definition
$
J_{\mathbf{i}}=[A_{\mathbf{i}},B_{\mathbf{i}}].
$


   We get the level $n$ children of an $x\in T$ with $H_{\mathbf{i}}(x)\in  J_{\mathbf{i}}$ as follows:

If $H_{\mathbf{ i}}^{-1 }(x)=\alpha + \frac{x-A_{\mathbf{i}}}{a^n}\not \in T$ then
 the level $n$ child of $x$  is $\Theta$. Otherwise, the level $n$ child of $x$ is $H_{\mathbf{ i}}^{-1 }(x)=\alpha+ \frac{x-A_{\mathbf{i}}}{a^n}$.

 \medskip

 \fm Note that $H_{i}^{-1 }(x)=\Phi_i(x)$, where $\Phi_i$ was defined in Definition \ref{z86}. In the sequel we will use the notation with the \fm$H_i^{-1}$ and their iterates.



 In general, if $x\in T$ and $\Z_0=\{x\}$ and  $A\subset T$ is a  Borel set, then for any $n\geq 1$ the set of level $n$ descendants of $x$ contained in  $A$ is denoted by $\mathfrak{D}_n(x,A)$. So,

\begin{equation}\label{eq:brpr}
\mathfrak{D}_n(x, A):=\left\{
\mathbf{i}\in\mathcal{L}_{n}: \ H_{\mathbf{ i}}^{-1 }(x)\in  A\right\}
\end{equation}
and
\begin{equation}
\label{y72}
\Z_n(x, A)= \# \mathfrak{D}_n(x, A).
\end{equation}
We remark that the process $(\mathcal{Z}_n)_{n=0}^\infty$ is a Markov chain since an
individual in $\mathcal{Z}_n$ gives birth to descendants independently of the individuals of the
same generation if $\mathcal{Z}_{n-1}$ is given.


A major role in our analysis is played by the expectations $\ev[\mathcal{Z}_n(x,A)]$, for
$A\subset T$, $n\ge 1$. For $n=1$ we have
\begin{eqnarray}\label{eq:dph}
\ev[\mathcal{Z}_1(x,A)]&=&\int_\Omega\mathcal{Z}_1(x,A)\di\pr
=\notag \int_\Omega\sum_{i=1}^L \ind_{\{\Phi_i(x)\in A\}} \di\pr\\
&=&\sum_{i=1}^L\pr
(\Phi_i (x)\in A)=\sum_{i=1}^L\int_A\! \phi_{i}(x,y) \di y,
\end{eqnarray}
where $\phi_{i}(x,y)$ was defined in part (b) of Definition \ref{z86}.\fm
 It follows that $M_1(x,\cdot):=\ev[\mathcal{Z}_1(x,\cdot)]$ has a  kernel, given by
\begin{equation}\label{m-kernel}
m(x,y):=m_1(x,y)=\fm \sum_{i=1}^L \phi_{i}(x,y), \quad
(x,y)\in T\times T.
\end{equation}
Let for $n\ge 1$ and $A\subset T$ and $x\in T$,
\begin{equation*}\label{115}
M_n(x,A):=\ev[\mathcal{Z}_n(x,A)]
\end{equation*}
We remark that if $M_1$ has a kernel then $M_n$ also has a kernel. Let us write $m_n(x,\cdot)$
for the kernel of $M_n(x,\cdot)$.  That is
\begin{equation}
\label{y73}
M_n(x,A)=\int\limits_{y\in A} m_n(x,y) \di y.
\end{equation}

The branching structure of $\mathcal{Z}$ yields (see \cite[p.67]{Harris63})
\begin{equation*}\label{eq:rec-mn}
m_{n+1}(x,y)=\int_{T}m_n(x,z)m_1(z,y)\di z,
\end{equation*}
which was already introduced in \fm \eqref{y80}, \fm where one has to realize that in the notation we suppressed the dependence on $\varepsilon$ of the kernel function $m(\cdot,\cdot)$ in Section 6 and 7.


\subsection{Supercritical branching processes with uniformly positive kernel}

Part (2) of the Main Lemma asserts  that there exists an integer
$N_0$ such that $m_{N_0}$ is a uniformly positive and uniformly bounded function. That is,
there exist $a_{\min}$ and $a_{\max} $ such that for all $x ,y\in \T$ we have
\begin{equation}\label{cond:dens}
0<a_{\min}\leq m_{N_0}(x ,y)\leq a_{\max}<\infty \tag{\textbf{C1}}.
\end{equation}
We next consider the following two operators:
\begin{eqnarray}\label{10}
&F:\ \varphi(x)\mapsto& \int\limits_{\T}m_1(x ,y)\cdot \varphi(y)\,\di y, \quad x\in T
\\
&G:\ \psi(y)\mapsto& \int\limits_{\T}\psi(x)\cdot m_1(x ,y)\, \di x, \quad y\in T.\notag
\end{eqnarray}

These operators are closely related to the expectations of the branching process. Note in particular that
\begin{equation}\label{eq:FandE}
\ev[\Z_n(x,A)]=\int_{T}m_n(x,y)\ind_A(y)\di y=F^n(\ind_A(x))\fm.
\end{equation}
We cite the following theorem from \cite[Theorem 10.1]{Harris63}:

\begin{theorem}[Harris]\label{Harris:10.1}
It follows from {\rm (\ref{cond:dens})} that the operators in {\rm (\ref{10})} have a
common dominant eigenvalue $\rho $. Let $f$ and $g$ be the
corresponding eigenfunctions of the first and second operator in
{\rm (\ref{10})} respectively. Then the functions $f$ and $g$
are bounded and uniformly positive on $T$. Moreover, apart from a scaling, $f$ and $g$ are
the only non-negative eigenfunctions of these operators. Further, if we normalize
$f$ and $g$ so that $\int f (x)g (x)\di x=1$, which will be
henceforth assumed, then for all  $x,y \in T$ as $n\to\infty$
$$
\Big|\frac{m_n(x,y)}{\rho^n}-f (x)g(y)\Big|\le C_1\,f (x)g (y)\Delta^n,
$$
\noindent where the bound $\Delta<1$ can be taken independently of $x$ and $y$,
and the constant $C_1$ is independent of $x,y$ and $n$.
\end{theorem}

It follows from part (3) of the Main Lemma that the branching process $\mathcal{Z}$ is supercritical. That is the
 the Perron-Frobenius eigenvalue $\rho$ is greater than one:
\begin{equation}\label{cond:rho}
\rho >1\tag{\textbf{C2}}.
\end{equation}

\section{The proof of Theorem \ref{y79} assuming Main Lemma \ref{z80}}\label{sec:Th12}

Choose $\varepsilon _{{\rm\scriptscriptstyle{MAIN}}}$ as in \eqref{z41}, \fm and fix an $\varepsilon $
and an \fm $\eta\fm >0$ such that
\begin{equation}
\label{y64}
\fm 0<\varepsilon <\varepsilon _{{\rm\scriptscriptstyle{MAIN}}} \text{ and } 0<\eta+\varepsilon  <\varepsilon _{{\rm\scriptscriptstyle{MAIN}}}.
\end{equation}
\fm We write
\begin{equation}
\label{y78} T:=T(\varepsilon ) \text{  and } \fm W:=T(\varepsilon +\eta ).
\end{equation}

\begin{lemma}\label{lem:g_1g_2}
  Fix an  $\varepsilon\in(0,\varepsilon _{{\rm\scriptscriptstyle{MAIN}}}) $.
 Let $f$ be a $F$-eigenfunction, then   there exist
 $0< \eta< \varepsilon _{{\rm\scriptscriptstyle{MAIN}}}-\varepsilon $ and $C_0>0$  such that introducing the notation
  \begin{equation*}
  f_W:=f\cdot \ind_{W},
  \end{equation*}
  we have for any $ x\in T=T(\eps)$
  \begin{equation}\label{y63}
   Ff_W(x)>C_0 f(x).
  \end{equation}
  \end{lemma}
\begin{proof}
 By the fact that $f$ is the right eigenfunction of $F$ corresponding to the Perron-Frobenius eigenvalue $\rho $ we obtain that for every $\fm x\in T$ we have
 $$
 \int\limits _{T} m(x,y)\di y\geq\rho \frac{f(x)}{\max\limits_{z\in T}f(z)}.
$$
Hence, for all $x\in T$,
\begin{equation}
 \label{y76}
\fm Ff(x)=
 \int\limits _{T} m(x,y)f(y)\di y\geq
\rho\big(\min\limits_{z\in T}f(z) \big)^2/{\max\limits_{z\in T}f(z)} =:C^*>0.
 \end{equation}
 The fact that $C^*>0$ follows from Harris' Theorem  (Theorem \ref{Harris:10.1}).

 Using the \fm definition of $m$ in \eqref{m-kernel},\fm   \eqref{z82}, and the properties of $\widetilde{f}_i$ in \fm {\bf A2} we obtain that there exists a number $U$ such that
\begin{equation}
\label{y75}
0\leq m(x,y)\leq U,\quad \text{ for all } (x,y)\in T\times T.
 \end{equation}
We choose $\eta >0$ so small that the second inequality below holds:
\begin{equation}
\label{y74}
\int\limits_{T\setminus W}m(x,y)f(y)\di y \leq\fm \Leb(T\setminus W)\cdot U\cdot\max\limits_{z\in T}f(z) <C^*/2.
\end{equation}
Putting together \eqref{y76} and \eqref{y74} we obtain that
$$\fm Ff_W(x)= Ff(x)-
\int\limits_{T\setminus W}m(x,y)f(y)\di y>C^*/2.$$
\fm Hence for all $x\in T$
$$  Ff_W(x)>\frac{C^*}{2\max\limits_{z\in T}f(z)}f(x).$$
That is, \eqref{y63} holds with $C_0 :=C^*/2\max\limits_{z\in T}f(z)$.
\end{proof}

\medskip

\fm By induction we obtain from Lemma \ref{lem:g_1g_2} that for all $x\in T$ and integers $n\ge 1$
$$\fm F^nFf_W(x)\ge C_0 F^nf(x)=C_0\rho^nf(x).$$
\fm Dividing both sides by $f(x)$, we therefore obtain the following corollary.\\
(See (\ref{eq:FandE}) for the $"="$ in \eqref{308}).

\begin{corollary}\label{lem:find_r}
For any $n\geq 1$ we have
$\forall x\in \HH,\quad F^n\ind_{W}(x)>\fm C_0\rho^{n-1}.$\\
Consequently, there \fm exists a positive integer $r$ such that for each $n\geq r$
\begin{equation}\label{308}
  \forall x\in \HH,\quad F^n\ind_{W}(x)=\ev[ \Z_n(x,W)] >6L.
\end{equation}
\end{corollary}

\medskip

\noindent Our next lemma is a corollary to the Hoeffding inequality \fm \cite{Hoeffding}:
\begin{lemma}[Hoeffding]
 Assume $Y_1,\dots,Y_C$ are independent random variables
such that
for any $i=1,\dots,C$ we have $a_i\leq Y_i\leq b_i$ for some real numbers $a_i,b_i$. Let
$S_C=\sum_{i=1}^C Y_i$ and let $t$ be a positive real number. Then we have
\begin{equation*}
\pr (S_C-\ev\, S_C>t)\leq \exp\Big\{-\frac{2t^2}{\sum_{i=1}^C(b_i-a_i)^2}\Big\}.
\end{equation*}
 \end{lemma}

 The following statement, Lemma \ref{lem:Cramer} is a slight generalization of the easy part of the Cram\'er theorem.

\begin{lemma}\label{lem:Cramer}Let $C\in\mathbb{N}$ and
$Z_1,\dots,Z_C$ be a sequence of independent random variables such that $Z_i$ takes values from $[0,L^r]$
and $m_i=\mathbb{E}\left[Z_i\right]>6$ for all $i=1,\dots  ,C$.
Let $r$ be as in { \rm (\ref{308})}. There exists $0<\tau=\tau(r)<1$ such that for all $C\geq 1$
\begin{equation*}
 \pr \Big(Z_1+\dots +Z_C<2C \Big)\leq \tau^C.
\end{equation*}
\end{lemma}
\begin{proof}
Let $m_i:=\ev\, Z_i$ and in general $m_{x}:=\ev Z_{x}$. We have
\begin{eqnarray*}
 \pr \Big(Z_1+\dots +Z_C<2C \Big)=\pr \Big(m_1-Z_1+\dots +m_C-Z_C>\sum_{i=1}^C(m_i-2)\Big).
\end{eqnarray*}
Now, we want to apply the Hoeffding inequality with
\begin{equation*}
 Y_i=m_i-Z_i,\quad \mbox{ and }\quad  t=\sum_{i=1}^C(m_i-2).
\end{equation*}
First, note that by the definition of $r$ in (\ref{308}) we have $m_{x}>6$ for any $x\in \HH$.
Therefore, $t>0$. Further, $\ev S_C=0$, where we remind the reader that $S_C=\sum_{i=1}^C Y_i$.
It is straightforward that $\left\{ Y_i=m_1-Z_i \right\} _{i=1}^{C}$
is contained in an interval of length $ L^r$.
Thus, we have
\begin{multline*}
\pr \Big(m_1-Z_1+\dots +m_C-Z_C>\sum_{i=1}^C(m_i-2)
\Big)\leq          \exp\Big\{\!-\frac{2\Big( \sum_{i=1}^C(m_i-2)\Big)^2}{\sum_{i=1}^C( L^r)^2}\Big\}\leq
\\
\exp\Big\{-\frac{2\Big(
\sum_{i=1}^C(6-2)\Big)^2}{C\cdot L^{2r}}\Big\}=\exp\Big\{-\frac{32C^2}{C\cdot L^{2r}}\Big\}=\Big(\exp\Big\{-\frac{32}{L^{2r}}\Big\}\Big)^C,
\end{multline*}
where we used $m_x>6$. Since $\tau=\exp\Big\{-\frac{32}{L^{2r}}\Big\}<1$, this proves the Lemma.
\end{proof}

\begin{definition}\label{y68}
 Let us denote by $c_1$ the length of the smallest interval in $T(0)$. We choose $n_1$ such that $a^{n_1}\fm\Leb(W)\approx c_1$, \fm specifically,
$n_1:=\ds\Big\lceil\log_a\frac{c_1}{\Lebfrac(W)}\Big\rceil$.
Let  $\ell_1$ be the length  of  the smallest interval in $W$.
\end{definition}

\begin{lemma}\label{lem:34}
  For any fixed $n\geq 1$ and for every $\omega \in\Omega$   (where $\Omega $ is the sample space defined in Section \fm \ref{Probspace})
  there exists a (random)  interval   $\J=\J(\omega ) \subset T(0)$  of length $\Leb(\J)=\frac12\ell_1 a^n$ such that for any $x\in \J$
  \begin{equation*}
  \Z_n(x,W)(\omega )\geq \frac{\Lebfrac(W)}{2(\beta -\alpha )}(La)^n=:N(n).
  \end{equation*}
  \end{lemma}
  \begin{proof}

  The proof uses the  observation that for any bounded integrable function $h$
  \begin{equation*}
  \int_{T(0)} h(x)\di x  \le  \Leb(T(0))\cdot||h||_\infty\le (\beta -\alpha )||h||_\infty.
  \end{equation*}

  By the definition of $\Z_n(x,W)$  (see (\ref{eq:brpr})) we have
  \begin{eqnarray*}
  \int_{T(0)} \Z_n(x,W)\di x &=& \int_{T(0)}   \sum_{|\vect{i}|=n}\ind_{H_{\mathbf{ i}}^{-1 }(x)\in W}   \di x\\
  &=&    \sum_{|\vect{i}|=n} \int_{T(0)}    \ind_{H_{\mathbf{ i}}^{-1 }(x)\in W}   \di x\\
  &=&    \sum_{|\mathbf{i}|=n}\Leb(\{x\in T(0):  H_{\mathbf{ i}}^{-1 }(x)\in W\})  \\
  &=&    \sum_{|\mathbf{i}|=n}  \Leb\left( T(0)\cap H_{\mathbf{ i}}(W) \right)    \\
  &=&    L^{n}\Leb(W)a^n=\fm 2(\beta -\alpha )N(n),
  \end{eqnarray*}
  where we use in \fm the step before the last step that $\Heta_{\mathbf{i}}(W)\subset T(0)$.
  This follows from {part (5) of the Main Lemma} since $W\subset T(0)$.

  In this way, for every $\omega \in\Omega $ there exists an $\xm=\xm(\omega )\in T(0)$ such that
  \begin{equation*}
  \Z_n(\xm,W) \geq 2N(n).
  \end{equation*}
Let $G_n:=\left\{ \mathbf{i}\in\mathcal{L}_{n}:  H_{\mathbf{ i}}^{-1 }(x_{\max})\in W \right\}$.
Then by definition $\# G_n = \Z_n(\xm,W)$. For each $\mathbf{i}\in G_n$, $H_{\mathbf{i}}^{-1}(\xm)$ is contained
in a connected component $C_{\mathbf{i}}$ of $W$. By definition $\Leb(C_{\mathbf{i}})\geq \ell _1$.

We write $G _{n}^{l }$  for the collection of those $\mathbf{i}\in G_n$ for which the center of $C_{\mathbf{i}}$ is to the left  from $H_{\mathbf{i}}^{-1}(\xm)$.
So, for an $\mathbf{i}\in G _{n}^{l }$ there is an interval of length at least $\ell _1/2$, contained in $W$ with right endpoint $H_{\mathbf{i}}^{-1}(\xm)$.
Let $G _{n}^{r}:= G_n\setminus G _{n}^{l}$. Then at least one of the sets $G _{n}^{l}$ or $G _{n}^{r}$ (say $G _{n}^{l  }$) has cardinality at least $N(n)$.
This means that for  $J:=\left(\xm-a^n\ell_1/2 , \xm \right)$ and for every $x\in J$ and for every $\mathbf{i}\in  G _{n}^{l  }$ we have
$H _{\mathbf{i}}^{-1}(x)\in W$. So, by \eqref{eq:brpr}, we have  $\mathcal{Z}_n(x,W)\geq \# G _{n}^{l  }=N(n)$ for all $x\in J$.
To verify that $J\subset T(0)$,  pick an  $\mathbf{i}\in G _{n}^{l }$. Then $J\subset H_{\mathbf{i}}(C_{\mathbf{i}})$.
Using that  $C_{\mathbf{i}}\subset W\subset T(0)$ and $H_{\mathbf{i}}(T(0))\subset  T(0)$ (see Fact \ref{y70}) yields $J\subset  T(0)$.
\end{proof}

We partition each interval of $T(0)$ into intervals of equal length. If $I$ is a connected component interval of
$T(0)$ then we partition it into $  {\widetilde{L}:=}\left\lceil 6\frac{|I|}{\ell_1a^{n}}\right\rceil$ subintervals.
In this way we obtain a partition of $T(0)$ into the intervals  $J_1,\dots,J_{\widetilde{L}}$ (labelled in increasing order)
such that for any \fma $\ell$,  $\Leb(J_{\ell}) \leq \Leb(\J)/3$, where $\J$ was defined in Lemma \ref{lem:34}.

For $\omega\in\Omega$  let $k=k(\omega )\in\{1,\dots,\widetilde{L}\}$ be chosen such that
\begin{equation*}
J_{k}\subset \fma \J,\quad J_{k-1}\nsubseteq \J,
\end{equation*}
where \fma $\J=\J(\omega )$ is the interval defined in Lemma \ref{lem:34}. Let
\begin{equation*}
\Omega_l=\{\omega: \fm k(\omega)=l\}
\end{equation*}
Note that, by Lemma \ref{lem:34} we have
\begin{equation*}
 \Omega=\bigcup_{l=1}^{\widetilde{L}}\Omega_l.
\end{equation*}

We define the sequence
\begin{equation*}
 a_k(n)=(L)^{n+kr}\tau^{2^{k-1}\cdot N(n)}.
\end{equation*}
For fixed $\xi>0$ let $n_2\geq n_1$ be chosen such that for any $n\geq n_2$ we have $\frac{\ell _1}{\eta } \leq (La)^n$ (where $\eta $ was defined in Lemma \ref{lem:g_1g_2}),
 and for any $n\geq n_2$ and for any $k\geq 0$
\begin{equation}\label{a99}
 a_k(n)<\frac{1}{2}\ \hbox{ and }\ \sum_{k=0}^\infty a_k(n)<\xi/2.
\end{equation}

\begin{lemma}\label{lem:last}
Fix an arbitrary $n\geq n_2$ and $l\in\{1,\dots,\widetilde{L}\}$. Then
\begin{equation}\label{eq:5}
\forall x\in J_l \hbox{ and } \forall \omega \in\Omega_l, \quad \Z_n(x,\HH)(\omega)>N(n).
\end{equation}
Further,
\begin{equation}\label{eq:6}
 \pr\left(\Z_{n+Mr}(x,\HH)>2^M\! \cdot N(n),M=0,1,\dots,\ \forall x\in J_l\mid \Omega_l
 \right)>\prod_{k=0}^\infty (1-a_k(n)).
\end{equation}

\end{lemma}
\begin{proof}
Equation (\ref{eq:5}) follows from Lemma \ref{lem:34}.

Concerning (\ref{eq:6}). Let $X_k$ be a $\eta a^{2n+kr}$ dense set in $J_l$, where $\eta$ has been set in Lemma \ref{lem:g_1g_2}. $X_k$ can be chosen such that
\begin{equation}
\label{y49}
\#X_k\leq\frac{\ell_1 a^n}{\eta a^{2n+kr}}=\frac{\ell_1}{\eta}a^{-(n+kr)}<L^{n+kr} \text{ if } n\geq n_2.
\end{equation}

\begin{fact}\label{y59}
 It follows from  Lemma \ref{lem:Cramer}  that for any $x\in J_l$
\begin{equation}\label{eq:7}
\pr \left( \Z_{n+r}(x,W)\leq 2 N(n) \mid \Omega_l\right)\leq \tau^{N(n)}.
\end{equation}
\end{fact}

\fm Before we prove  Fact \ref{y59} we need  to define $H_{\mathbf{i}\to\mathbf{i}\mathbf{j}}(x)$, in the following Fact.

\medskip
\begin{fact}\label{y62}
Let $\mathbf{i}\in\mathcal{L}_{n}$ and $\mathbf{j}\in\mathcal{L}_{m}$ for some $n,m>0$.
We define the random function $ H_{\mathbf{i}\to\mathbf{i}\mathbf{j}}$ on $T$ by \fm
  \begin{equation}
  \label{y61}
  H_{\mathbf{i}\to\mathbf{i}\mathbf{j}}(x):= a^mx+t_{\mathbf{j}}+Y _{j_1}^{\mathbf{ i}}+\sum _{\ell =2}^{m}a^{\ell -1}
  Y _{j_{\ell }}^{\mathbf{i}(\mathbf{j}|_{\ell -1} )},
  \end{equation}
  where $\mathbf{i}(\mathbf{j}|_{\ell -1} ):=i_1\dots i_nj_1\dots j_{\ell-1}$.
  Then $H_{\mathbf{i}\mathbf{j}}(x)=H_{\mathbf{i}}\circ H_{\mathbf{i}\to\mathbf{i}\mathbf{j}}(x)$. \\Consequently,
  \fm $H_{\mathbf{i}\mathbf{j}}^{-1}(x) = H_{\mathbf{i}\to\mathbf{i}\mathbf{j}}^{-1}\circ H_{\mathbf{i}}^{-1} (x)$,   and
\begin{equation}
\label{y60}
\fm H_{\mathbf{i}\to\mathbf{i}\mathbf{j}}^{-1}(x)=   \frac{x}{a^m}-\frac{t_{\mathbf{j}}}{a^m} -
\left( Y _{j_1}^{\mathbf{ i}}+ \sum _{k=2}^{m }\frac{1}{a^{m+1-k}}Y _{j_k}^{\mathbf{i}(\mathbf{j}|_{k-1}) }\right).
\end{equation}
\end{fact}

\medskip

\begin{proof}[Proof of Fact \ref{y59}]
  Let $x\in T$ and $U\subset T$.
First we recall \fm that $\mathfrak{D}_n(x,U)$  was defined in \eqref{eq:brpr}.
Using this for an $\mathbf{i}\in \mathfrak{D}_n(x,T)$ and $r\geq 1$  now we define
\begin{equation}
\label{y58}
\mathfrak{D}_{r,\mathbf{i}}(x,U):=\left\{\mathbf{j}\in\mathcal{L}_r: H_{\mathbf{i}\to\mathbf{i}\mathbf{j}|_{k-1}}^{-1}(x_{\mathbf{i}}) \fm \in T,
k\leq r-1, H^{-1}_{\mathbf{i}\to\mathbf{i}\mathbf{j}}(x_{\mathbf{i}}) \fm \in U \right\}
\end{equation}
where $x_{\mathbf{i}}:=H^{-1}_\mathbf{i}(x)$, and
\begin{equation}
\label{y57}
\mathfrak{D}_{r,\mathbf{i}}^q(x,U):=
\left\{
  \mathbf{j}\in \mathfrak{D}_{r,\mathbf{i}}(x,U):
j_1=q
 \right\}.
\end{equation}
Then by definition
\begin{equation}
\label{y56}
\mathcal{Z}_{n+r}(x,W)=\sum _{\mathbf{ i}\in \mathfrak{D}_n(x,T)} \sum _{q=1}^{L}\# \mathfrak{D}_{r,\mathbf{i}}^q(x,W).
\end{equation}
Observe that for any $x\in T$ and $\mathbf{i}\in \mathfrak{D}_n(x,T)$
\begin{equation}
\label{y55}
Z_r(x_{\mathbf{i}},W)
\stackrel{d}{=}
\sum _{q=1}^{L}\# \mathfrak{D}_{r,\mathbf{i}}^q(x,W).
\end{equation}
So by Corollary \ref{lem:find_r} there is a (random) $q=q(x,\mathbf{i},\omega )\in [L]$ such that
\begin{equation}
\label{y54}
\mathbb{E}\left[\# \mathfrak{D}_{r,\mathbf{i}}^q(x,W)\right]>6.
\end{equation}
For an $x\in T$ and for every $\mathbf{i}\in \mathfrak{D}_n(x,T)$, we write $Y _{\mathbf{i}}:= \# \mathfrak{D}_{r,\mathbf{i}}^q(x,W)$.
Observe that the random variables
\begin{equation}
\label{y53}
\left\{ Y_{\mathbf{i}} \right\} _{\mathbf{ i}\in \mathfrak{D}_n(x,T)} \quad \text{are independent.}
\end{equation}

\fm
Namely, \fm \fm according to \eqref{y60} in \fm Fact \ref{y62} for a $\mathbf{j}\in [L]^r$ with $j_1=q$  the random part of $H_{\mathbf{i}\to\mathbf{i}\mathbf{j}}^{-1}$
 is $Y _{q}^{\mathbf{ i}}+  \sum\limits_{k=2}^{r }\frac{1}{a^{m+1-k}}  Y _{j_k}^{\mathbf{i}qj_2\dots \fm\fm j_{k-1} }$.
 This implies that for distinct  $\mathbf{i},\widehat{\mathbf{i}} \in \mathfrak{D}_n(x,T) $
and for $\mathbf{j},\widehat{\mathbf{j}}\in [L]^r$ with $j_1=q=q(\mathbf{i})$  and
$\widehat{j}_1=q=q(\widehat{\mathbf{i}})$ the random variables
$H_{\mathbf{i}\to\mathbf{i}\mathbf{j}}^{-1}$ and $H_{\widehat{\mathbf{i}}\to\widehat{\mathbf{i}} \widehat{\mathbf{j}}}^{-1}$ are independent.
Hence we get that \eqref{y53} holds. Finally, $Y_{\mathbf{i}}$ takes values from $[0,L^r]$ and by \eqref{y56} we have
\begin{equation}
\label{y52}
\mathcal{Z}_{n+r}(x,W)\geq \sum _{\mathbf{ i}\in \mathfrak{D}_n(x,T)}
Y_{\mathbf{i}}.
\end{equation}
  {For an $x\in J$} let
\begin{equation}
\label{y50}
C:=\mathcal{Z}_{n}(x,T)=\#\mathfrak{D}_n(x,T) \geq \mathcal{Z}_{n}(x,W)  {\geq} N(n)
\end{equation}
Then by virtue of Lemma \ref{lem:Cramer} and the Markov property  there exists a $\tau \in(0,1)$ (where $\tau $  depends only on $r$)
  {such that for all $x\in J_l$ we have:}
\begin{equation}
\label{y51}
\pr \left( \Z_{n+r}(x,W)\leq 2 N(n) \mid \Omega_l\right)\leq \pr \Big(Y_1+\dots +Y_C<2C\  |\  \Omega _l \Big)\leq
\mathbb{E}\left[\tau^C  |\  \Omega _l \right]<\tau ^{N(n)}.
\end{equation}
\end{proof}
To prove equation (\ref{eq:6}) we will use induction. More precisely, we will prove that the inequality
\begin{equation}\label{eq:66}
  \pr\left(\Z_{n+kr}(x,\HH)>2^k\! \cdot N(n),\forall 0\leq k\leq M,\ \forall x\in J_l\mid
  \Omega_l \right)>\prod_{k=0}^M (1-a_k(n)).
\end{equation}
holds for any positive integer $M$.

For $M=1$, by (\ref{eq:7}) we obtain:
$$
\mathbb{P}\left(\exists x\in X_1,\ \Z_{n+r}(x,W)\leq 2N(n)|\Omega _l \right)\leq \# X_1\cdot \tau ^{N(n)}.
$$
Recall that $X_1$ was defined as an $\eta a^{2n+r}$-dense subset of $J_l$.
 Recall that by \eqref{y49} we have \fmu $\# X_1\leq L^{n+r}$. Hence,
\begin{equation}\label{eq:11}
 \pr \left( \Z_{n+r}(x,W)> 2 N(n),x\in X_1 \mid \Omega_l\right)\geq  1-L^{n+r}\tau^{N(n)}.
\end{equation}
Next, our purpose is to extend the inequality (\ref{eq:11}) from all $x\in X_1$ to all $x\in J_l$.
Let us fix $k\geq 1$.

We will use the following fact.
\begin{fact}\label{fact:meas}
For any $k\geq 1$ and $l=1,\dots,\widetilde{L}$ we have
\begin{equation}\label{eq:subset}
\left\{\Z_{n+kr}(x,W)> 2^k N(n),x\in X_k \right\} \cap \Omega_l\subset \left\{\Z_{n+kr}(x,\HH)> 2^k N(n),x\in J_l \right\} \cap \Omega_l,
\end{equation}
and the set
\begin{equation}\label{eq:meas}
\left\{\forall x\in J_l:\ \Z_{n+kr}(x,\HH)\geq 2^kN(n) \right\}
\end{equation}
is measurable.
\end{fact}

\begin{proof}[Proof of Fact \ref{fact:meas}]
Using the definition of
$W=T(\eps+\eta)$ and $\HH=T(\eps)$, if for some $x'$ and $\omega$ one has
$\Z_{n+kr}(x',W)> 2^k N(n)$, then for the same $\omega$ and for any $x$ such that
$|x-x'|<\eta a^{n+kr}$ and larger set $\HH$ we also have $\Z_{n+kr}(x,\HH)> 2^k N(n)$.
Further, since $X_k$ is $\eta a^{2n+kr}$ dense, for any $x\in J_l$ we can find $x'\in X_k$ such that
$|x-x'|<\eta a^{2n+kr}<\eta a^{n+kr}$. This proves (\ref{eq:subset}).

It  remains to be proved that the set in (\ref{eq:meas}) is measurable which is formally not straightforward since $x$ is running over an interval $J_l$.

First, we note that it is enough to prove that for any fixed $x'\in X_k$ the set
\begin{equation*}
 \left\{\forall x\in \left[x'- \eta a^{n+kr},x'+ \eta a^{n+kr}\right]\cap J_l:\  \Z_{n+kr}(x,\HH)\geq 2^kN(n) \right\}
\end{equation*}
is measurable since $X_k$ is a finite set.

We have to take into consideration two facts. $\HH$ is a union of finite number of intervals and according to \eqref{y72}
$\Z_{n+kr}(x,\HH)$ is a sum of a finite number of indicator functions:
\begin{equation*}
 \Z_{n+kr}(x,\HH)=
 \#\mathfrak{D}_{n+k}(x,T)
\end{equation*}
Therefore, the function $\Z_{n+kr}(\cdot ,\HH)$ for any $\omega$ is a jump function on
$T$ with \fmu a finite number of jumps. Let $\{\iota_i:i\in \mathcal{I}\}$ denote the partition of $T$
into the intervals on which $\Z_{n+kr}(\cdot ,\HH)$ is constant. So, $\Z_{n+kr}(x ,\HH)$ depends
on the interval $\iota_i$ which $x$ falls into. Therefore,
\begin{multline*}
 \left\{\forall x\in \left[x'- \eta a^{n+kr},x'+ \eta a^{n+kr}\right]\cap J_l:\
 \Z_{n+kr}(x,\HH)\geq 2^kN(n) \right\}=
\\
\left\{\forall i\in \mathcal{I}\hbox{ such that } \iota_i\cap \left[x'- \eta a^{n+kr},x'+ \eta
a^{n+kr}\right]\cap J_l\neq\emptyset:\ \Z_{n+kr}(\iota_i,\HH)\geq 2^kN(n) \right\}.
\end{multline*}
The last set is given by a measurable function of a finite number of random variables
 hence measurable.
\end{proof}

As a consequence of Fact \ref{fact:meas} we can exchange $X_1$ with $J_l$ at a price of replacing the set $W$ with the larger set $T$ in \eqref{eq:11}.
In this way we  obtain
\begin{equation*}
 \pr \left( \Z_{n+r}(x,\HH)> 2 N(n),\forall  x\in J_l \mid \Omega_l\right)\geq
 1-L^{n+r}\tau^{N(n)}=1-a_1(n).
\end{equation*}
For $k=0$, using Lemma \ref{lem:34}, the definition of $\Omega_l$ and $W\subset \HH$, we
have
\begin{equation*}
\Z_n(x,\HH)\geq \Z_n(x,W)\geq N(n) \quad \hbox{ for all
}\omega\in\Omega_l.
\end{equation*}
Therefore, we have (\ref{eq:6}) for $M=1$:
\begin{multline*}
 \pr \left( \Z_{n+kr}(x,\HH)> 2^k N(n),0\leq k\leq 1,x\in J_l \mid \Omega_l\right)\geq
 1-L^{n+r}\tau^{N(n)}=
\\
1-a_1(n)>(1-a_1(n))(1-a_0(n)).
\end{multline*}

\bigskip

Now, assume that we have proved (\ref{eq:66}) for $M-1$. We will prove it for $M$.
 The simple fact that for any three events $A,B,C$ of positive probability we have:  $\mathbb{P}(A\cap B|C)=\mathbb{P}\left(A|B\cap C\right)\cdot
\mathbb{P}\left(B|C\right)$ yields
\begin{multline*}
  \pr\left(\Z_{n+kr}(x,\HH)>2^k\! \cdot N(n),\forall 0\leq k\leq M,\ \forall x\in J_l\mid
  \Omega_l \right)=
\\
\pr\left(\Z_{n+Mr}(x,\HH)>2^M\cdot N(n),\ \forall x\in J_l\mid \Z_{n+kr}(x,\HH)>2^k\! \cdot
N(n),\forall 0\leq k\leq M-1,\right.
\\
\left.
\forall x\in J_l, \; \Omega_l \right)\cdot
\\
\cdot\pr\left( \Z_{n+kr}(x,\HH)>2^k\! \cdot N(n),\forall 0\leq k\leq M-1,\ \forall x\in J_l\mid
\Omega_l \right)
\end{multline*}
By induction, it is known that the second term on the right hand side is larger than
$\prod_{k=0}^{M-1}(1-a_k(n))$.
Now we use a similar argument, to the one
 applied as in the case $M=1$, for proving that the first term, on the right hand side in the displayed formula above, is larger than
$1-a_M(n)$. Namely, let $\Omega_{l,M-1}$ denote the event in the condition of the first
term:
\begin{equation*}
 \Omega_{l,M-1}:=\left\{\Z_{n+kr}(x,\HH)>2^k\! \cdot N(n),\forall 0\leq k\leq M-1,\ \forall
 x\in J_l, \right\}\cap \Omega_l.
\end{equation*}
As in the proof of (\ref{eq:7}) we use Lemma \ref{lem:Cramer}.  Let $n'=n+(M-1)r$,
$C=\Z_{n+(M-1)r}(x,\HH)$, then for any $x\in J_l$ we have
\begin{equation*}
\pr \left(\Z_{n'+r}(x,W)\leq 2\cdot 2^{M-1} N(n)\mid \Omega_{l,M-1} \right)\leq \tau
^{2^{M-1}N(n)}.
\end{equation*}
This can be proved in exactly the same way as (\ref{eq:7}) was proved. The continuation is also
similar, we first take a dense set $X_M$ and prove the counterpart of (\ref{eq:11}), that is,
\begin{equation*}
 \pr \left( \Z_{n'+r}(x,W)> 2 \cdot 2^{M-1}N(n),x\in X_M \mid \Omega_{l,M-1}\right)\geq
 1-L^{n+Mr}\tau^{2^{M-1}N(n)}=1-a_M(n).
\end{equation*}
Applying Fact \ref{fact:meas} again yields
\begin{equation*}
 \pr\left(\Z_{n+Mr}(x,\HH)>2^M\cdot N(n),\ \forall x\in J_l\mid \Omega_{l,M-1}\right)\geq
 1-a_M(n)
\end{equation*}
using $\Z_{n+Mr}(x,\HH)=\Z_{n'+r}(x,\HH)$. This finishes the proof of Lemma \ref{lem:last}.
\end{proof}

Now, we are ready to present the proof of Theorem \ref{y79}\label{page:end_of_proof}.

\begin{proof}[Proof of Theorem \ref{y79} assuming the Main Lemma]
Using Lemma \ref{lem:last} we have
\begin{multline*}
 \pr \left(C_{\mathcal{H}}\hbox{ contains an interval }\mid \Omega_l\right)\geq \pr \left(
 C_{\mathcal{H}} \hbox{ contains }J_l\mid \Omega_l\right)\geq
\\
\pr\left(\Z_{n+Mr}(x,\HH)>2^M\! \cdot N(n),\forall M,\ \forall x\in J_l\mid \Omega_l
\right)>\prod_{k=0}^\infty (1-a_k(n))>1-2\sum_{k=0}^\infty a_k(n)
\end{multline*}
Getting rid of the condition, we obtain that
\begin{multline*}
 \pr \left(C_{\mathcal{H}} \hbox{ contains an interval }\right)=
\\
\sum_{l=1}^\infty   \pr \left(C_{\mathcal{H}} \hbox{ contains an interval }\mid \Omega_l\right)\pr
(\Omega_l)>
\\
\left(1-2\sum_{k=0}^\infty a_k(n) \right)\sum_{l=1}^L\pr (\Omega_l)=1-2\sum_{k=0}^\infty
a_k(n)>1-\xi,
\end{multline*}
  {where in the last step we used \eqref{a99}.}
Since $\xi$ can be chosen arbitrarily small this proves Theorem \ref{y79}.

\end{proof}

\section{Construction of the pre-typespace $T(0)$ and the  type space $T(\varepsilon )$ }\label{y81}
In this section we prove our Main Lemma \ref{z80}.

We consider the support of $m_I$:
$$
\mathrm{supp}(m_I)=\bigcup\limits _{i=1}^{L}
\left\{ (x,y): y\in \mathrm{supp }\  \Phi _{i}(x)
  \right\}.
$$
It is immediate that
\begin{equation}
\label{z79}
\mathrm{supp}(m_I)=
\bigcup\limits _{i=1}^{L}
\widehat{S}_i,
\end{equation}
where,
$$
\widehat{S}_i:=
\left\{ (x,y):
x\in \widehat{W}_i,\
y\in \mathrm{supp}\ \Phi _i(x)
\right\} \text{ for }
\widehat{W}_i:=(a\alpha +t_i-\theta_i,a\beta  +t_i+\theta_i).
$$
It is easy to see that for all $i\in[L]$, $\widehat{S}_i$ is a parallelogram with two horizontal sides: $\left\{(x,y): y=\alpha  \right\}$, $\left\{(x,y): y=\beta  \right\}$
and the two other sides are the following two lines of slope $1/a$
$$
\widehat{\ell }_i^2(x):=\frac{1}{a}x-\frac{t_i}{a}+\frac{1}{a}\theta_i,\quad
\widehat{\ell }_i^1(x):=\frac{1}{a}x-\frac{t_i}{a}-\frac{1}{a}\theta_i.
$$
That is

\begin{eqnarray}
\label{y71}
\widehat{S}_k &=&
\left\{ (x,y): x\in \widehat{W}_k,\
\max\left\{ \alpha ,\widehat{\ell }_k^1(x) \right\}< y < \min\left\{ \beta ,\widehat{\ell }_k^2(x) \right\}\right\} \\ \nonumber
&=& \left\{ (x,y): y\in (\alpha ,\beta ), ay+t_i-\theta _i<x<ay+t_i+\theta _i \right\}.
\end{eqnarray}

Clearly,
\begin{equation}
\label{z46}
\mathrm{width}(\widehat{S}_k)=2\theta_k\quad
\mathrm{height}(\widehat{S}_k)=\frac{2\theta_k}{a}.
\end{equation}

In general the open filled parallelograms $\widehat{S}_i$ are not disjoint.
Their union $\bigcup\limits _{i=1}^{L}\widehat{S}_i$ has say $M$ connected components
$\left\{ S_k \right\}_{k=1}^M$. By elementary geometry, for all $k\in[M]$ the connected component $S_k$ is also an open filled parallelogram having two horizontal sides and the Western non-horizontal sides is one of the lines from
$\left\{ \widehat{\ell }_i^2(x) \right\}_{i=1}^{L } $. Let us call it
$\ell _{k }^{2 }(x)$. While the Eastern non-horizontal side of $S_k$ is one of the lines from $\left\{ \widehat{\ell }_i^1(x) \right\}_{i=1}^{L } $. Let us call it
$\ell _{k }^{1 }(x)$.

\begin{figure}[ht!]
  \begin{center}
   \includegraphics[width=7cm]{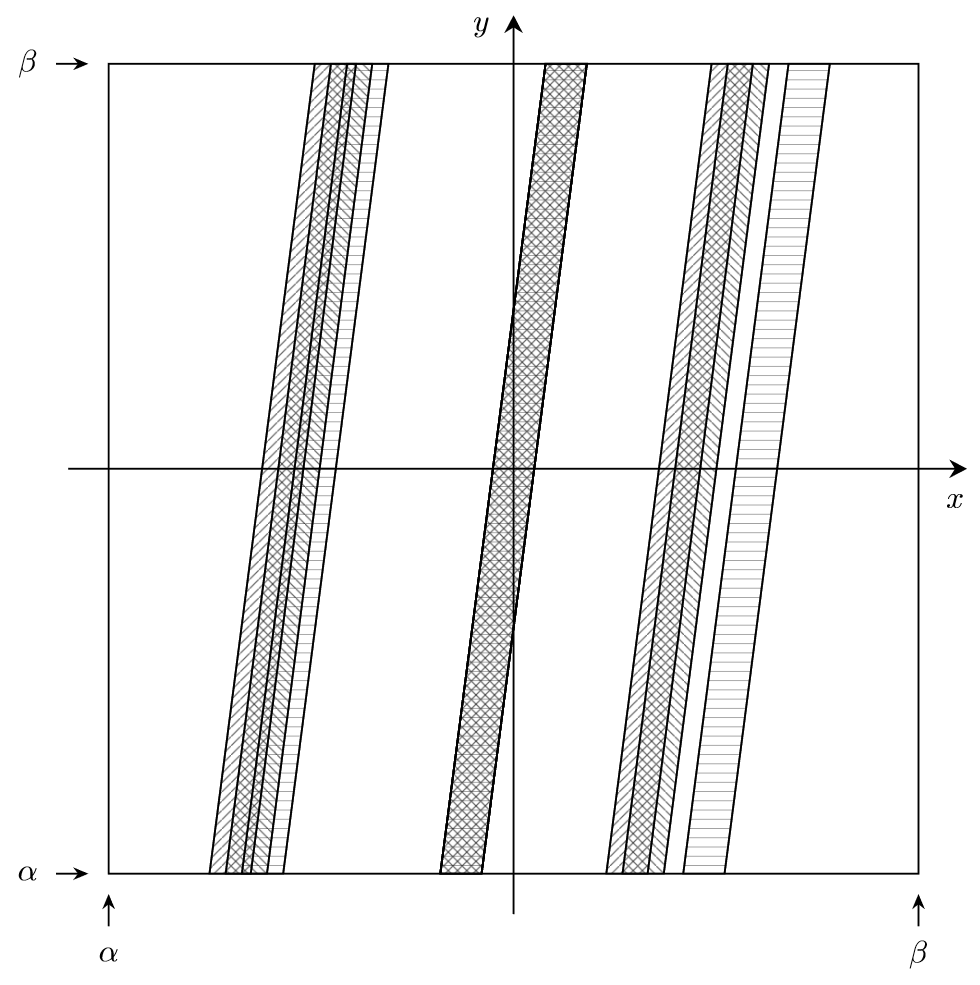}
  \includegraphics[width=7cm]{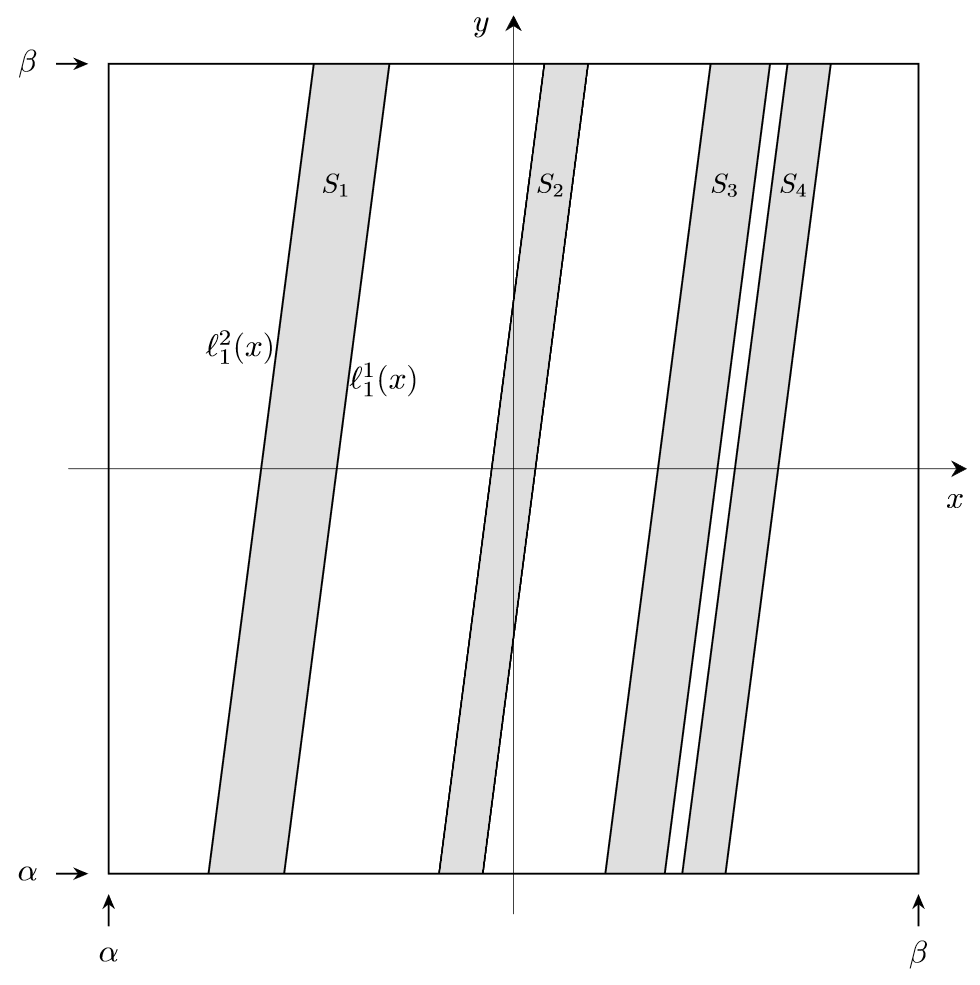}
  \end{center}
  \caption{Paralellograms $\widehat{S}_i$ on the left and paralellograms
  $S_k$ on the right.}\label{y92}
\end{figure}

Observe that the open filled parallelograms $S_k$ can be adjacent to each other. That is, it can happen that
for a $1\leq k\leq  M-1$, we have $\ell _{k}^{1 }(x)\equiv \ell _{k+1}^{2 }(x)$.

We introduce the orthogonal projections to the coordinate axes:
$$
\pi _1(x,y):=x \quad \text{ and } \quad \pi _2(x,y):=y.
$$
We have
\begin{equation}
\label{z78}
S_k=
\left\{
(x,y): x\in \pi _1(S_k),\quad
\max\left\{ \alpha ,
\ell_k^1(x) \right\}
< y <
\min\left\{
\beta ,
\ell _k^2(x)
 \right\}
 \right\}.
\end{equation}
The collection of parallelograms having two horizontal sides such that the slopes of the other two sides are equal to $1/a$ is denoted by $\mathfrak{P}$.
 Particular attention will be given to the  parallelograms of the form $P_{\langle u,v\rangle}^k\in \mathfrak{P}$, where $u\leq v$,  and
 $\langle a,b\rangle$ stands for an interval with endpoints $a\leq b$ about which we do not know if it is closed or open or half-closed and half-open.
\begin{equation}
\label{z77}
P_{\langle u,v\rangle}^k:=\left\{ (x,y)\in S_k: y\in\langle u,v\rangle  \right\}.
\end{equation}
When $u=v$ then $P_{[ u,v]}^k$ is the horizontal line segment $\left\{ (x,y)\in S_k: y=u=v \right\}$.
 Without loss of generality, we may assume  that $\bigcup\limits _{k=1}^{M}S_k$ is contained in the region between
the Western side of $S_1$ (which is determined by the graph of the function
$\ell  _{1}^{2}(x)$)
 and the Eastern side of $S_M$ (which is determined by the graph of the function $\ell _{M}^{1}(x)$).

We define $\alpha <\widetilde{\alpha } <\widetilde{\beta }<\beta $
by
\begin{equation}
\label{z76}
 \ell _{1}^{2 } (\widetilde{\alpha } )=\widetilde{\alpha },\quad
 \ell _{M}^{1 } (\widetilde{\beta } )=\widetilde{\beta }\quad \text{ and } \quad
 \widetilde{I}:=[\widetilde{\alpha },\widetilde{\beta }]\subset I.
\end{equation}
\begin{definition}\label{z68}
  Let $\mathfrak{A}$ be the collection of all finite unions of
  sub-intervals of $\widetilde{I}$, including all
  open, closed, half-open and half-closed, and even degenerated intervals.
  In particular  \fma there exists  a $q$ such that
$H\in \mathfrak{A}$, $H=\bigcup\limits_{j=1}^q \langle a_i,b_i \rangle $, where
$\widetilde{\alpha }\leq a_i \leq b_i\leq \widetilde{\beta }$ for all $i\in [q]$.
  We define
\begin{equation}
\label{z64}
U_{H}:= \bigcup\limits_{k\in[M]}\bigcup\limits_{i\in[q]}
P_{\langle a_i,b_i\rangle}^k
=
\bigcup\limits_{i\in[q]}
\left\{ (x,y): y\in \langle a_i,b_i\rangle \right\}
\bigcap
\bigcup\limits _{k=1}^{M}S_k.
\end{equation}
and
  \begin{equation}
\label{z67}
\Psi (H):=
\pi _1\left(
  U_H
 \right)=
\bigcup\limits_{k\in[M]}\bigcup\limits_{i\in[q]}
\pi _1 \left( P_{\langle a_i,b_i\rangle}^k \right).
\end{equation}
  \end{definition}
  \begin{claim} \label{z65}
   The mapping $\Psi $ \fm satisfies
   \begin{enumerate}
   [{\bf (a)}]
   \item Let $y\in \widetilde{I}$. Then
    \begin{equation}
   \label{y80extra}
   \Psi \left( \left\{ y \right\} \right)=
  \bigcup\limits_{i=1}^{L}
  \left( ay+t_i-\theta _i,ay+t_i+\theta _i  \right).
   \end{equation}
     \item \fm $\Psi$ is a self-mapping of $\mathfrak{A}$. That is,  $\Psi : \mathfrak{A}\to \mathfrak{A}$.
     \item \fm $\Psi$ is monotone in the sense that $A\subset B$ implies $\Psi (A)\subset \Psi (B)$.
     \item If $H\in \mathfrak{A}$ is open then so is $\Psi(H)$.
   \end{enumerate}
  \end{claim}

Namely, part (a) follows from the second part of  \eqref{y71}.
it is clear that for all $k\in[M]$ and $i\in[L]$,  $\pi _1 \left( P_{\langle a_i,b_i\rangle}^k \right)$ and so
$\Psi (H)$ consists of finitely many intervals.
The fact that all of these intervals are contained in $\widetilde{I}$
follows from the definition of $\widetilde{\alpha }$ and $\widetilde{\beta }$.
Part (b) and (c) are obvious from the definition.

Set
\begin{equation}
\label{z66}
V_m:=\Psi^{m} (\mathrm{int} (\widetilde{I}))\quad \text{ with }\quad
V_0:=\mathrm{int}( \widetilde{I}).
\end{equation}
  {It follows from}
 part (b) {and (c)} of Claim \ref{z65}, {that}  $V_{m+1}\subset  V_{m}$. Put
\begin{equation}
\label{z49}
N_0:=
\inf\left\{ m\in\mathbb{N}\cup\left\{ \infty   \right\}:
V_{m}\setminus V_{m+1}=\emptyset
\right\}.
\end{equation}
We will prove in Claim \ref{z15} that $N_0$ is finite.
\begin{definition}\label{z63}
  The pre-type space is defined by $T(0):=V_{N_0}$.
\end{definition}

\subsection{Elementary properties of the pre-type space $T(0)$}

Using that both $V_m\in \mathfrak{A}
$ and $\mathrm{int}(\widetilde{   I})\setminus V_m \in \mathfrak{A}$
for every $m$ we can find
 an $\widehat{n}_m$, $n_m$ and $\widetilde{\alpha }\leq \alpha _{i}^{(m) }<\beta  _{i}^{(m) }\leq\widetilde{\beta }$ and $\widetilde{\alpha }\leq
u_i^{(m)}\leq v_i^{(m)}
 \leq\widetilde{\beta }
 $ such that
$$
V_m= \bigcup\limits_{i\in [\widehat{n}_m]}
\left( \alpha _{i}^{(m) },\beta  _{i}^{(m) } \right) ,\quad
V_{m-1}\setminus V_m
=
\bigcup\limits_{i\in [n_m]}
\langle  u_i^{(m)}, v_i^{(m)}  \rangle
$$
and $\left\{\left( \alpha _{i}^{(m) },\beta  _{i}^{(m) } \right)  \right\}
_{i=1}^{\widehat{n}_m }$ and
$
\left\{
  \langle  u_i^{(m)}, v_i^{(m)}  \rangle
 \right\}_{i=1}^{n_m }
$
are the connected components of $V_m$ and
$V_{m-1}\setminus V_m$ respectively.
We say that the intervals in the first union are level $m$
green intervals and the intervals in the second union are the level $m$ red intervals.
See Figure \ref{y91}.
Their  collections  are denoted  by $\mathcal{G}_m$ and $\mathcal{R}_m$ respectively. That is
$$
\mathcal{G}_m:=\left\{
   \left(\alpha _i^{(m)}, \beta _i^{(m)}  \right)
 \right\}_{ i=1}^{\widehat{n}_m }
\text{ and }
\mathcal{R}_m:=\left\{ \langle  u_i^{(m)}, v_i^{(m)}
  \rangle \right\}_{ i=1}^{n_m }.
$$

\begin{figure}[ht!]
  \centering
  \includegraphics[width=10cm]{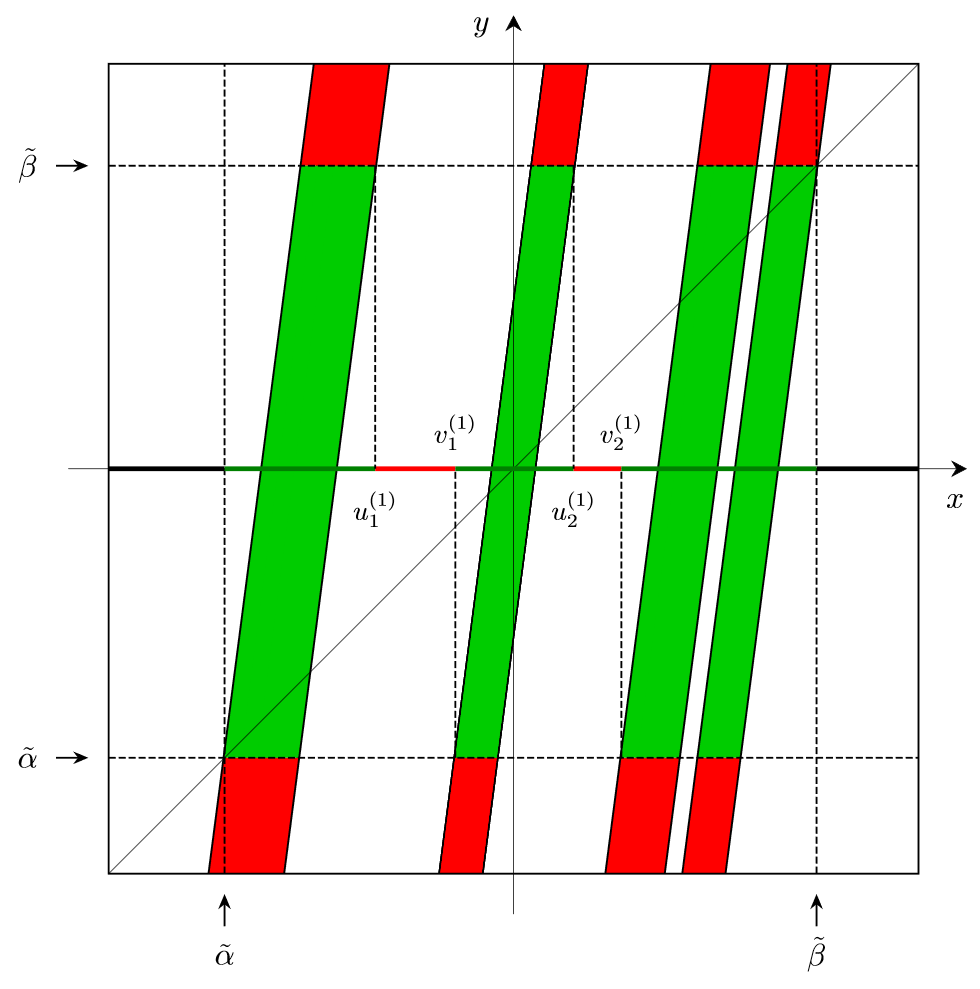}
  \caption{Level 1 red and green intervals.}\label{y91}
\end{figure}

For an $m\geq 1$,
 the level-$m$ green and red areas are:
\begin{equation}
\label{z57}
G_m:=U_{V_{m-1}}=
\bigcup\limits_{i\in [\widehat{n}_{m-1}],k\in [M]}
P _{ \left(  \alpha _{i}^{(m-1) },\beta  _{i}^{(m-1) } \right)}^{k },\
R_m:=U_{V_{m-1}\setminus  V_m}=
\bigcup\limits_{i\in [n_m],k\in [M]}
P _{ \langle  u_i^{(m)}, v_i^{(m)}  \rangle}^{ k}.
\end{equation}
Hence,
$$
\pi _1(G_m)=V_m,\quad
R_m\subset  G_m, \quad \text{and} \quad
G_{m+1}=G_m\setminus R_m.
$$
In words, we get level $m+1$ green area if we take away level $m$ red from level $m$ green area. This implies that the sets $\left\{ R_i \right\}_i$
are pairwise disjoint and
\begin{equation}
\label{z62}
G_{m+1}=G_1\setminus \bigsqcup\limits _{i=1}^{m } R_{\ell },
\end{equation}
where $\bigsqcup$ means disjoint union.
In words: we get the level-$m+1$ green area if we take aways from the level $1$ green area all the first $m$ level red areas.
That is
 \begin{equation}
\label{z73}
G_{m+1}=G_1\setminus \bigcup\limits _{\ell =1}^{m }
\bigcup\limits_{1\leq i\leq n_{\ell }, k\in[M]}
{P }_{\langle u_i^{(\ell )}, v_{i }^{(\ell ) } \rangle}^k.
\end{equation}
The following Lemma plays an important role in our proofs.
\begin{lemma}\label{z74}
 $\pi _1(G_{m-1})\setminus \pi _1(G_{m})$ is the union of disjoint closed (possibly degenerated) intervals
$\mathcal{R}_m:= \left\{ [u _{i}^{(m) },v _{i}^{(m) }] \right\}_{i=1}^{n_m } $.
\end{lemma}
This follows from elementary geometry by mathematical induction.
\begin{proof}[Proof of the Lemma]
As we have mentioned,  it is immediate from the construction that for every $\ell $
the open set $\pi _1(G_{\ell })\subset (\widetilde{\alpha },\widetilde{\beta })$ has finitely many components.
It follows from \eqref{z76}
that if $\pi _1(G_1)\setminus \pi _1(G_2) $ is not empty then it is the disjoint union  of finitely many
compact intervals. Assume that the same holds for an $m\geq 2$
for the non-empty $\pi _1(G_{m-1})\setminus \pi _1(G_{m})$.
That is
\begin{equation}
\label{z56}
\pi _1(G_{m-1})\setminus \pi _1(G_{m})
=
V_{m-1}\setminus V_m
=
\bigcup\limits_{i\in [n_m]}
\left[ u_i^{(m)}, v_i^{(m)} \right].
\end{equation}
Clearly, this is the case if and only if all
 endpoints of $\pi _1(G_{m-1})$
are also endpoints of $\pi _1(G_{m})$. That is our induction hypotheses yields that for all $i\leq n_m$ and $j\leq \widehat{n}_m$
\begin{equation}
\label{z69}
  \left( u _{i}^{(m) },v _{i}^{(m) } \right)
\subset
\left( \alpha _{j}^{(m-1) }, \beta _{j}^{(m-1) } \right) \Longrightarrow
\left[ u _{i}^{(m) },v _{i}^{(m) } \right]
\subset
\left( \alpha _{j}^{(m-1) }, \beta _{j}^{(m-1) } \right).
\end{equation}
Using this, we  prove that whenever
$\pi _1(G_{m})\setminus \pi _1(G_{m+1})$ is non-empty then it is the union of finitely many disjoint closed intervals.
To verify this we assume that
\begin{equation}
\label{z55}
x\in \pi _1(G_{m})\setminus \pi _1(G_{m+1})\subset \pi _1\left( G_m\setminus G_{m+1} \right)=\pi _1(R_m).
\end{equation}
Then by \eqref{z57} there is a $k\in [M]$ and $i\in [n_m]$, $j\in[\widehat{n}_m]$
and
$\left[ u _{i}^{(m) },v _{i}^{(m) } \right]\in\mathcal{R}_m$,
$\left(  \alpha _{j}^{(m-1) },\beta  _{j}^{(m-1) } \right)\in
\mathcal{G}_{m-1}
$
such that by
\eqref{z55} and
the induction hypothesis we have
\begin{equation}
\label{z53}
x\in
\pi _1\left( P _{ \left[ u _{i}^{(m) },v _{i}^{(m) } \right]  }^{k } \right)
\subset
\pi _1 \left( P _{ \left(  \alpha _{j}^{(m-1) },\beta  _{j}^{(m-1) } \right)}^{k } {\setminus \mathcal{G}_{m+1}}\right).
\end{equation}
Using that $\left[ u _{i}^{(m) },v _{i}^{(m) } \right]\subset \left( \alpha _{j}^{(m-1) }, \beta _{j}^{(m-1) } \right) \subset  V_{m-1} $ is a connected component of $V_{m-1}\setminus V_m$, there exists an $\varepsilon >0$
such that
\begin{equation}
\label{z54}
\left( u _{i}^{(m) }-\varepsilon ,u _{i}^{(m) } \right)\subset V_m
\cap
\left( \alpha _{j}^{(m-1) }, \beta _{j}^{(m-1) } \right)
,
\left( v_{i}^{(m)}, v_{i}^{(m)}+\varepsilon  \right)\subset V_m
\cap \left( \alpha _{j}^{(m-1) }, \beta _{j}^{(m-1) } \right).
\end{equation}
Putting together this, \eqref{z53}, \eqref{z55} and \eqref{z57} we obtain that
\begin{equation}
\label{z52}
x\in
\pi _1\left(
P _{\left[ u _{i}^{(m) },v _{i}^{(m) } \right]  }^{ k}\right)
\setminus
\left(
  \pi _1\left(  P _{ \left( u _{i}^{(m) }-\varepsilon ,u _{i}^{(m) } \right) }^{ k}\right)
\cup
\pi _1\left(
P _{\left( v_{i}^{(m)}, v_{i}^{(m)}+\varepsilon  \right) }^{ k}
 \right)\right)
\end{equation}
Then by simple elementary geometry this implies that the distance between $x$ and the boundary of $\pi _1 \left(P _{\alpha  _{j}^{(m-1)},\beta  _{j}^{(m-1) } }^{k } \right)$
is at least $\theta_{\min}$. Having a look at formula
\eqref{z57} we can see that $x$ cannot be an endpoint of
a component of $\pi _1(G_m)$. Hence,  any \fmu endpoint of any component of $\pi _1(G_m)$ is also an endpoint of a certain component of $\pi _1(G_{m+1})$.
\end{proof}
Actually we proved a little more.
\begin{remark}\label{z48}
  Note that as a by-product of the proof above, we obtain that the following assertion holds:

  If
$x\in \pi _1(G_{m})\setminus \pi _1(G_{m+1})$ then there exist some
$\left[ u _{i}^{(m) },v _{i}^{(m) } \right]\in\mathcal{R}_m$ and $k\in[M]$
such that
$x\in\pi _1\left( P _{ \left[ u _{i}^{(m) },v _{i}^{(m) } \right]  }^{k }\right)$. In this case the distance between $x$ and the boundary of
$\pi _1\left( P _{ \left[ u _{i}^{(m) },v _{i}^{(m) } \right]  }^{k }\right)$
is at least
$$
\theta_{\min}:=\min_{i\in [L]}\theta_i.
$$
\end{remark}
\begin{claim}\label{z51}
    The intervals in the collection
  $$
\mathcal{R}:=\left\{ \left[ u _{i}^{(m) },v _{i}^{(m) } \right] \right\}_{m\in[N_0],i\in[n_m]}
  $$
  are pairwise disjoint.
\end{claim}
\begin{proof}
Consider the distinct intervals $\left[ u _{i}^{(m) },v _{i}^{(m) } \right],
 \left[ u _{j}^{(\widetilde{m}) },v _{j}^{(\widetilde{m}) } \right]\in\mathcal{R}$. If $m=\widetilde{m}$ then $\left[ u _{i}^{(m) },v _{i}^{(m) } \right]\cap
 \left[ u _{j}^{(\widetilde{m}) },v _{j}^{(\widetilde{m}) } \right]=\emptyset  $
follows from the fact that these are connected components of
$V_{m-1}\setminus V_m$. If $m\leq\widetilde{m}-1$
then $\left[ u _{i}^{(m) },v _{i}^{(m) } \right]\cap V_m=\emptyset $ but
$\left[ u _{j}^{(\widetilde{m}) },v _{j}^{(\widetilde{m}) } \right]\subset V_m$. So, $\left[ u _{i}^{(m) },v _{i}^{(m) } \right]\cap
\left[ u _{j}^{(\widetilde{m}) },v _{j}^{(\widetilde{m}) } \right]=\emptyset  $
also {holds} in this case.
\end{proof}

\begin{claim}\label{z15}
The number $N_0  $, defined in \eqref{z49}, is finite.
\end{claim}
\begin{proof}
Assume that $m\geq 2$ and $x\in V_m\setminus V_{m+1}$. Then we can choose
$\left[ u _{i}^{(m) },v _{i}^{(m) } \right]\in\mathcal{R}_m$ and $k\in[M]$,
$x\in\pi _1\left( P _{ \left[ u _{i}^{(m) },v _{i}^{(m) } \right]  }^{k }\right)$. Let
$$
v':=(\ell  _{k}^{2 })^{-1}(v^{(m)}_i) \text{ and }
u':=(\ell _{k}^{1})^{-1}(u_i^{(m)}).
$$
By definition, we can choose a $j\in[n_{m+1}]$ such that
$x\in \left[ u _{j}^{(m+1) },v _{j}^{(m+1) } \right]$. It follows from
Remark \ref{z48} and elementary geometry that
\begin{equation}
\label{z47}
v _{j}^{(m+1) }- u _{j}^{(m+1) } \leq v'-u' \leq
a\left( v _{i}^{(m) }-u _{i}^{(m) } \right)-\theta_{\min}.
\end{equation}
Let $\widetilde{g}(x):= ax-\theta_{\min}$. Then the length of the maximal interval in $\mathcal{R}_m$ is at most
$\widetilde{g}^m(\widetilde{\beta }-\widetilde{\alpha })$, where $\widetilde{g}^m$ is the $m$  fold iterate of $\widetilde{g}$.
However, $\widetilde{g}^m(\widetilde{\beta }-\widetilde{\alpha })<0$  if $m$ \fma is large enough.
The largest $m$ for which $\widetilde{g}^m(\widetilde{\beta }-\widetilde{\alpha })\geq 0$ is an upper bound on $N_0$.
\end{proof}

Recall that we defined $T(0)$ as $V_{N_0}$. It is clear that
$T(0)\ne \emptyset $ because by the construction $V_m\ne \emptyset $ for all $m$.

\begin{definition}\label{z45}
\begin{enumerate}
[{\bf (a)}]
\item The minimal width and height of the stripes $\left\{ S_k \right\}_{k=1}^{M }$ are denoted by $w$ and $h$. Clearly, $w\geq 2\theta_{\min}$ and $h\geq \frac{2\theta_{\min}}{a}$.
\item  To shorten the notation we write $\tau:=\widehat{n}_{N_0}$, and instead of
$
\left\{ (\alpha _{i}^{ (N_0)},\beta _{i}^{(N_0) }) \right\} _{i=1 }^{\widehat{n}_{N_0} }
$
we write $\left\{ (\alpha _i,\beta _i) \right\}_{i=1}^{\tau }$ for the connected components of $T(0)=V_{N_0}$.
\item The intervals $\left\{ (a_{i,k},b_{i,k}) \right\}_{k\in[M], i
\in[\tau]}$ are defined as follows:
$$(a_{i,k},b_{i,k}):=\pi _1\left( P _{ (\alpha _i,\beta _i)}^{ k} \right).$$
\item Set $\widehat{d}:=\min \left\{ |x-y|:
x\ne y,
x,y\in\bigcup\limits_{k\in[M],i\in[\tau]} \left\{ a_{i,k},b_{i,k} \right\}
\right\}$.
\end{enumerate}
\end{definition}
\begin{claim}\label{z44}
  If $x\in T(0)\cap \bigcup\limits_{k\in[M],i\in[\tau]} \left\{ a_{i,k},b_{i,k} \right\}$ then
  there is an $i(x)\in[\tau]$ and $k(x)\in[M]$ such that
\begin{equation}
\label{z43}
\widehat{d}\leq
x-a_{i(x),k(x)}  \text{ and }
\widehat{d}\leq
b_{i(x),k(x)}-x.
\end{equation}

\end{claim}
\begin{proof}
  Without loss of generality we may assume that
$x=a_{\widehat{i},\widehat{k}}$.
  We know that
$T(0)=V_{N_0}=V_{N_0+1}$. Hence,
\begin{equation}
\label{y66}
x\in T(0)=\Psi(T(0)) =
\bigcup\limits_{k\in[M],i\in[\tau]}
\pi _1\left(
  P _{(\alpha _i,\beta _i)}^{k }
 \right)
 =
 \bigcup\limits_{k\in[M],i\in[\tau]}
 (a_{i,k},b_{i,k}).
\end{equation}
So, there is an $i(x)\in[\tau]$ and a $k(x)\in[M]$ such that
$x\in (a_{i(x),k(x)},b_{i(x),k(x)})$. Then by the definition of $\widehat{d}$,
\eqref{z43} holds.
\end{proof}

\begin{fact}\label{y70}
  \begin{equation}
  \label{y69}
  y\in T(0) \Longrightarrow
  \bigcup\limits_{i=1}^{L}
  (ay+t_i-\theta _i,ay+t_i+\theta _i)\subset T(0).
  \end{equation}
  Consequently,
\begin{equation}
\label{y67}
H_{\mathbf{i}}(T(0))\subset T(0),\quad
\forall n\geq 1 \text{ and }
\mathbf{i}\in\mathcal{L}_{n}.
\end{equation}
\end{fact}
\begin{proof}
  The implication in \eqref{y69} follows from  \eqref{y80extra} and from the fact
 that $T(0)=V_{N_0}=V_{N_0+1}=\Psi (V_{N_0})$.
 Using this, and the definition of $H_i$ we get that  the image \fmu $H_{i}(y)$ of an $y\in T(0)$ by the random mapping
 $H_{i}$ satisfies $H_{i}(y) \in (ay+t_i-\theta _i,ay+t_i+\theta _i)$ for all $i\in [L]$.
 Successive applications of this inclusion  yields \eqref{y67}.
\end{proof}

\begin{lemma}\label{z39}
  There is an $\widetilde{\varepsilon} >0$ such that
for all $0<\varepsilon <\widetilde{\varepsilon }$  the Perron Frobenius eigenvalue of the operator $F^{\varepsilon }$ is greater than $1$.
\end{lemma}

The proof can be obtained by obvious modifications from the proof of \cite[Lemma 8A]{dekking2011algebraic}.

\begin{definition}[Type space]\label{z42}\
\begin{enumerate}
[{\bf (a)}]
  \item First we define
  \begin{equation}
  \label{z41}
  \varepsilon _{{\rm\scriptscriptstyle{MAIN}}}:=
  \frac{a}{10}\min\left\{
w,\widehat{d},\min\limits_{i\in[\tau]}(\beta _i-\alpha _i),
\widetilde{\varepsilon  }
   \right\}.
  \end{equation}
  \item
Fix an arbitrary  $0< \varepsilon < \varepsilon _{{\rm\scriptscriptstyle{MAIN}}}$. The type space is defined by
\begin{equation}
  \label{z40}
  T(\varepsilon ):=
  \bigcup\limits_{i\in[\tau]}\left[
  \alpha _i+\varepsilon ,\beta _i-\varepsilon
   \right].
  \end{equation}
\end{enumerate}
\end{definition}

\begin{claim}\label{z37}
  For an $x_0\in T(\varepsilon )$ we define
  \begin{equation}
  \label{z36}
  E _{1}^{\varepsilon }(x_0):=\left\{
    y\in T(\varepsilon ):
    m^{\varepsilon }(x_0,y)>0
   \right\}=
   \left\{ (x,y):x=x_0 \right\}\cap \mathrm{supp}\left( m _{1}^{\varepsilon } \right).
  \end{equation}
  Then there exists a $\kappa=\kappa (\varepsilon )>0$ such that for all $x_0\in T(\varepsilon )$
  the set $E _{1}^{\varepsilon }(x_0)$ contains an interval of length $\kappa$.
\end{claim}

\begin{proof}
  First recall that $T(0)=V_{N_0}=\pi _1\left( G_{N_0} \right)$
  and $G_{N_0}=\bigcup\limits_{k\in[M],i\in[\tau]}P _{(\alpha_i,\beta _i)}^{ k}$  and  $(a_{i,k},b_{i,k})=\pi _1\left( P _{ (\alpha _i,\beta _i)}^{ k} \right)$. Now we define the hexagon
  \begin{equation}
  \label{z35}
  H_{i,k}:=P _{(\alpha_i,\beta _i)}^{ k}
  \cap
  \left\{ (x,y):
  a_{i,k}+\varepsilon \leq x\leq  b_{i,k}-\varepsilon,
  \alpha _i+\varepsilon \leq y \leq
  \beta _i-\varepsilon
  \right\}.
  \end{equation}
  See Figure \ref{y97}.
  Clearly,
  \begin{equation}
  \label{z31}
  \mathrm{supp}(m _{1}^{\varepsilon  })=
  \bigcup\limits_{k\in[M],i\in[\tau]}
  H_{i,k}.
  \end{equation}

\begin{figure}[ht!]
  \includegraphics[height=9cm]{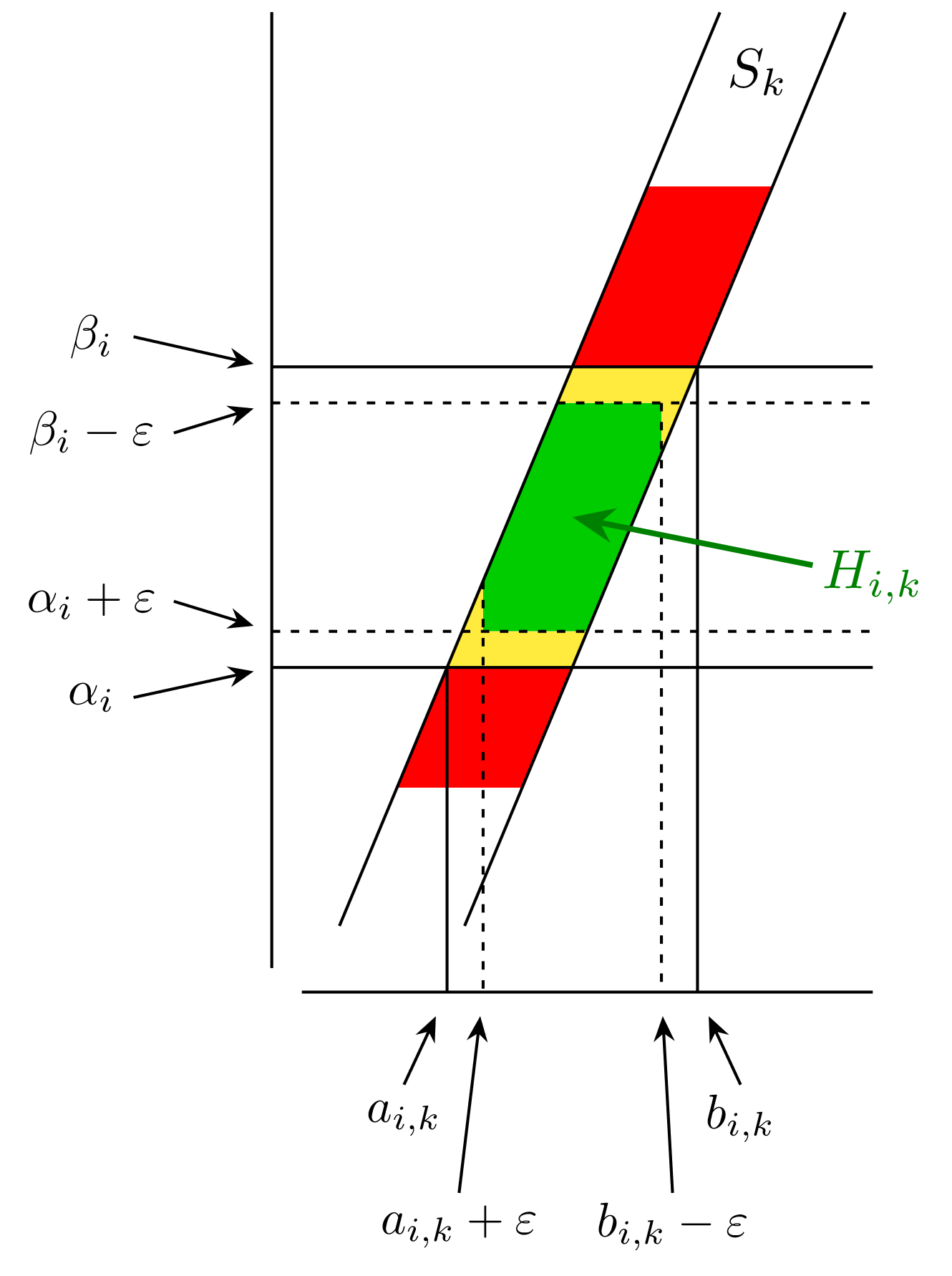}
  \caption{The definition of $H_{i,k}$.}
  \label{y97}
  \end{figure}

Observe that by elementary geometry
\begin{equation}
\label{z34}
\forall k\in[M],\  i\in[\tau],\  x_0\in \pi _1(H_{i,k}) \text{ we have }
|\left\{ (x,y):x=x_0 \right\}\cap
H_{i,k}
|\geq \varepsilon \left( \frac{1}{a} -1\right).
\end{equation}
Now we verify that
\begin{equation}
\label{z33}
T(\varepsilon )\subset
\bigcup\limits_{k\in[M],i\in[\tau]}
\pi _1
\left(
 H_{i,k}
 \right).
\end{equation}
Namely, $T(\varepsilon )\subset T(0)$ and in this way
for all $x_0\in T(\varepsilon )$
there exists a
$k\in[M]$ and an $i\in[\tau]$ such that
$x_0\in \pi _1(P _{(\alpha_i,\beta _i)}^{ k})=(a_{i,k},b_{i,k})$. Observe that
\begin{equation}
\label{z32}
a_{i,k}+\varepsilon  \leq x_0 \leq  b_{i,k}-\varepsilon
\Longrightarrow
x_0\in\pi _1\left( H_{i,k} \right).
\end{equation}
If the condition of \eqref{z32} does not hold then we may assume, without loss of generality that
\begin{equation}
\label{z30}
a_{i,k}< x_0< a_{i,k}+\varepsilon.
\end{equation}
 This and $x_0\in T(\varepsilon )$ imply that $a_{i,k}\in T(0)$. We argue by contradiction. If
$a_{i,k}\not\in T(0)$
then there is a  $j\in [\tau]$ such that $a_{i,k}=\alpha _j$. Then
$x_0\in (\alpha _j,\alpha _j+\varepsilon )\subset T(\varepsilon )^c$ which contradicts to our assumption that $x_0\in T(\varepsilon )$. So, we have verified that $a_{i,k}\in T(0)$.  Then
we can apply Claim \ref{z44} to conclude that there exists an $\ell \in[M]$ and
$j\in[\tau]$ such that
 $a_{j,\ell }+\widehat{d} <a_{i,k}<b_{j,\ell }-\widehat{d} $. Putting together this \eqref{z30} and \eqref{z41} we obtain that
$a_{j,\ell }+\varepsilon <x_0< b_{j,\ell }-\varepsilon $. As we have seen above this means that $x_0\in \pi _1(H_{j,\ell })$. This proves that \eqref{z33} holds. {Putting together
\eqref{z33},
\eqref{z34} and \eqref{z31} we get that the assertion of the Claim is true.}
\end{proof}
As a byproduct of the previous proof we obtain that

\begin{equation}
\label{y65}
T(\varepsilon )= \bigcup\limits_{k\in[M],i\in[\tau]} \pi _1 \left( H_{i,k} \right) =
 \bigcup\limits_{k\in[M],i\in[\tau]} \fmu [a_{i,k}+\varepsilon ,b_{i,k}-\varepsilon ].
\end{equation}
Namely, the non-trivial inclusion was verified above. The opposite inclusion is obvious by the definitions.

Our aim is to prove the following proposition which is actually Part (2) of the Main Lemma.

\begin{proposition}\label{z29}
  Fix an $0<\varepsilon <\varepsilon _{{\rm\scriptscriptstyle{MAIN}}}$. Then there exists an $n$
  such that for every $x_0\in T(\varepsilon )$   we have
  \begin{equation}
  \label{z28}
  E _{n}^{\varepsilon }(x_0):=   \left\{ (x,y): x=x_0 \right\}\cap  \mathrm{supp}(m _{n}^{\varepsilon  })
  = \left\{y: m _{n}^{\varepsilon  }(x_0,y)>0  \right\}=T(\varepsilon ).
  \end{equation}
\end{proposition}

\subsubsection{The structure of green and red areas}
To prove Proposition \ref{z29} we need to verify some auxiliary facts about the structure of the green and red areas and intervals.

For an $m\leq N_0$,
we write $\mathcal{B}_m$ and $\mathcal{J}_m$ for the collection of the left and right endpoints respectively, of the level $m$ red intervals (intervals from $\mathcal{R}_m$), and $\mathcal{B}=\cup_{i=1}^{N_0}\mathcal{B}_m$ and
$\mathcal{J}=\cup_{i=1}^{N_0}\mathcal{J}_m$.
It is immediate from the construction that the following \fmu statements hold:
\begin{fact}\label{z27}
\begin{enumerate}
[{\bf (a)}]
  \item For all $1\leq m\leq N_0$ and for all $(\alpha ',\beta ')\in \mathcal{G}_m$
there exists $v(\alpha ')\in\bigcup\limits _{\ell =1}^{m }\mathcal{J}_{\ell }$
$u(\beta ')\in \bigcup\limits _{\ell =1}^{m }\mathcal{B}_{\ell }$ such that $(\alpha ',\beta ')=(v(\alpha '),u(\beta '))$.
\item All elements of $\bigcup\limits _{\ell =1}^{m } \mathcal{B}_{\ell }$
 are \fm right endpoints of an element of $\mathcal{G}_m$. Similarly, all  elements of $\bigcup\limits _{\ell =1}^{m } \mathcal{J}_{\ell }$
 are \fm left endpoints of an element of $\mathcal{G}_m$.
\end{enumerate}
\end{fact}

Hence we get

\begin{fact}\label{z26}  
  Let $(\alpha ',\beta ')\in \mathcal{G}_m$. Then $(\alpha ',\beta ')\cap T(\varepsilon )=(\alpha '+\varepsilon ,\beta '-\varepsilon)\cap T(\varepsilon ) $.
\end{fact}
Namely, it follows from {Fact \ref{z27}} that every endpoint of a component of
$V_m$ is an endpoint of a component of
$T(0)$. See Figure \ref{y34}.

\begin{figure}[ht!]
    \includegraphics[height=7cm]{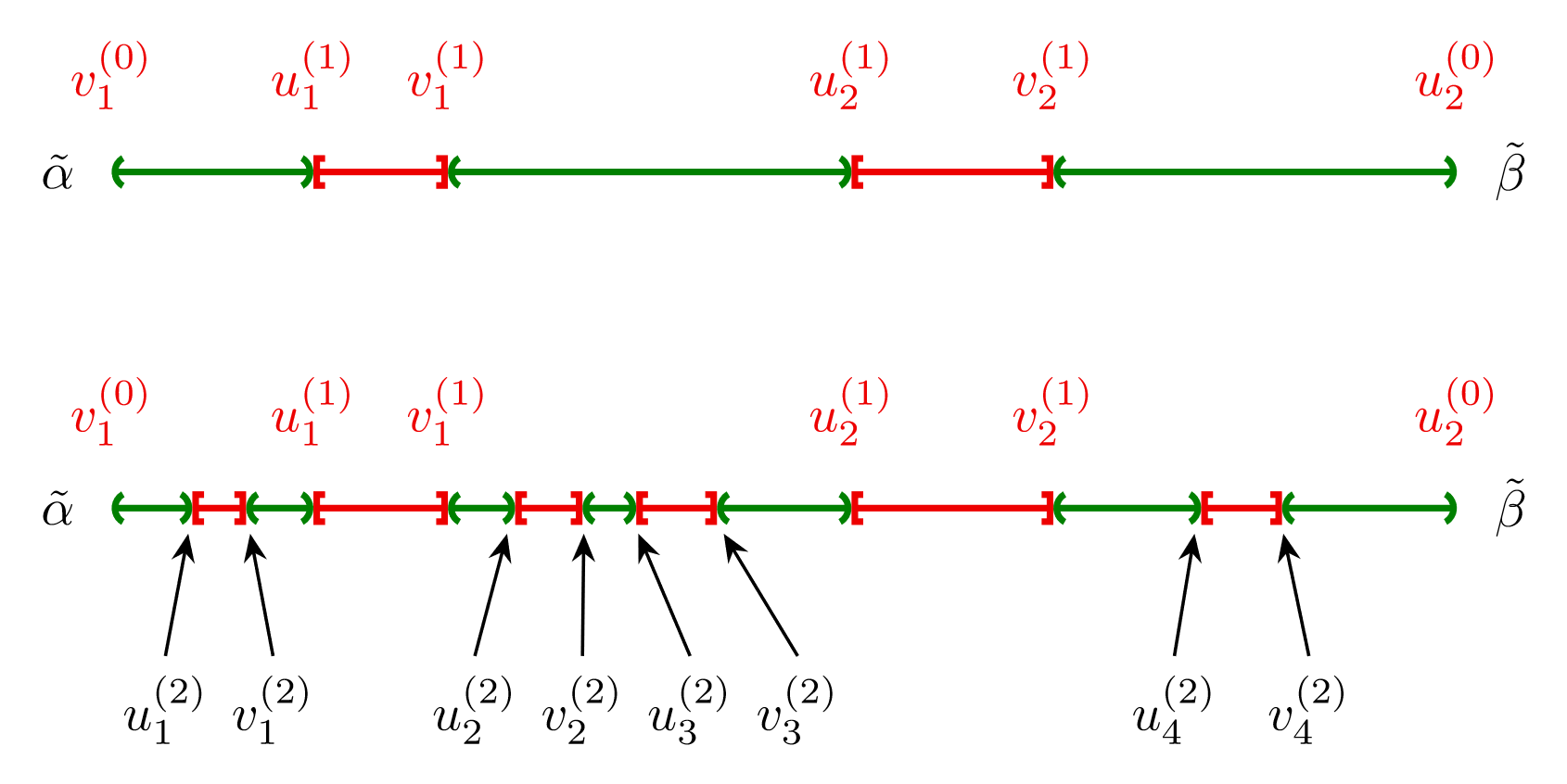}
    \caption{The red and green intervals}.
    \label{y34}
    \end{figure}

\begin{fact}\label{z23}  
  Let $z\in T(\varepsilon )$ and $k\in[M]$. Set
  $
  L(z,k):=\pi _1\left(
\left\{ (x,y):y=z \right\}\cap S_k
   \right).
   $
   Then
   \begin{equation}
   \label{z22}
   L(z,k)\cap T(\varepsilon )\ne \emptyset.
   \end{equation}
\end{fact}
\begin{proof}
 Fix a  $z\in T(\varepsilon )$ and $k\in[M]$.
 Observe that $L(z,k)\subset  T(0)$. Namely, $z\in T(0)=V_{N_0}$. So
 $\pi _1(L(z,k))\subset V_{N_0+1}=V_{N_0}=T(0)$. {Then} $\left(\mathcal{J} \cup \mathcal{B}  \right)\cap L(z,k)=\emptyset $.
 Hence, when we change from $T(0)$ to $T(\varepsilon )$ we can loose in a $1-1$ way intervals of length $\varepsilon $ each at the two ends of the interval of $L(z,k)$. Using that $|L(z,k)|=w_k\gg 2\varepsilon $ we get that
 $|L(z,k)\cap T(\varepsilon )|>w-2\varepsilon $, where $w$ is the minimum width of a stripe $S_k$.
\end{proof}
This implies that
\begin{equation}
\label{z21}
\forall z\in T(\varepsilon ), k\in[M]\  \exists
w\in T(\varepsilon )\cap L(z,k) \text{ such that }
z\in \left( \ell _{k}^{1}(w),\ell _{k}^{2}(w) \right).
\end{equation}

\fm 

\begin{figure}[ht!]
    \includegraphics[height=7cm]{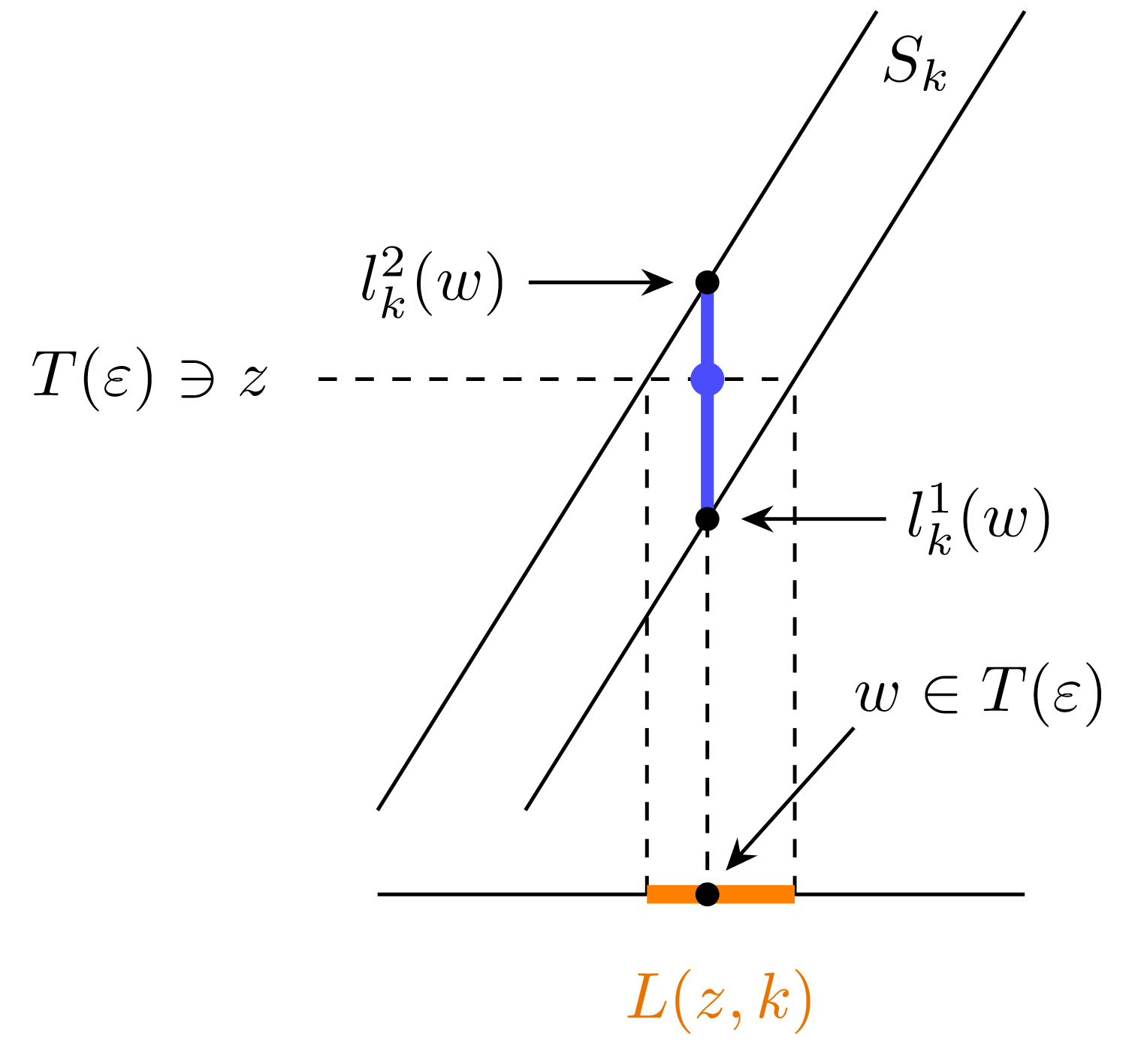}
    \caption{  {The notation of Fact \ref{z23}}. }
    \label{y96}
    \end{figure}

\begin{fact}\label{z25}  
 For an $1\leq m\leq N_0$ let $(\alpha ',\beta ')\in \mathcal{G}_m$.
 Consider the maximal collection of
    $(\alpha '_j,\beta '_j)\in\mathcal{G}_{m-1}$
    and $k(j)\in[M]$, $j=1,\dots  ,\ell $  such that for
$
P_j:= P _{(\alpha '_j,\beta '_j) }^{k(j) }
$ we have $\bigcup\limits_{j=1}^{\ell }
\pi_1\left(
P _j
 \right)\subset (\alpha ',\beta ').$
 See Figure \ref{y95}.
 For $j\ne j'$ if $(\alpha '_j,\beta '_j)=(\alpha '_{j'},\beta '_{j'})$ then
 $k(j)\ne k(j')$.
 We say that the intervals $ (\alpha '_j,\beta '_j) \in \mathcal{G}_{m-1}$, $j=1,\dots  ,\ell $ are the children of the interval $(\alpha ',\beta ')\in \mathcal{G}_m$. Then we have
\begin{equation}
\label{z24}
\bigcup\limits_{j=1}^{\ell }
\pi_1\left(
P _j
 \right)= (\alpha ',\beta ').
\end{equation}

\begin{figure}[ht!]
    \includegraphics[height=7cm]{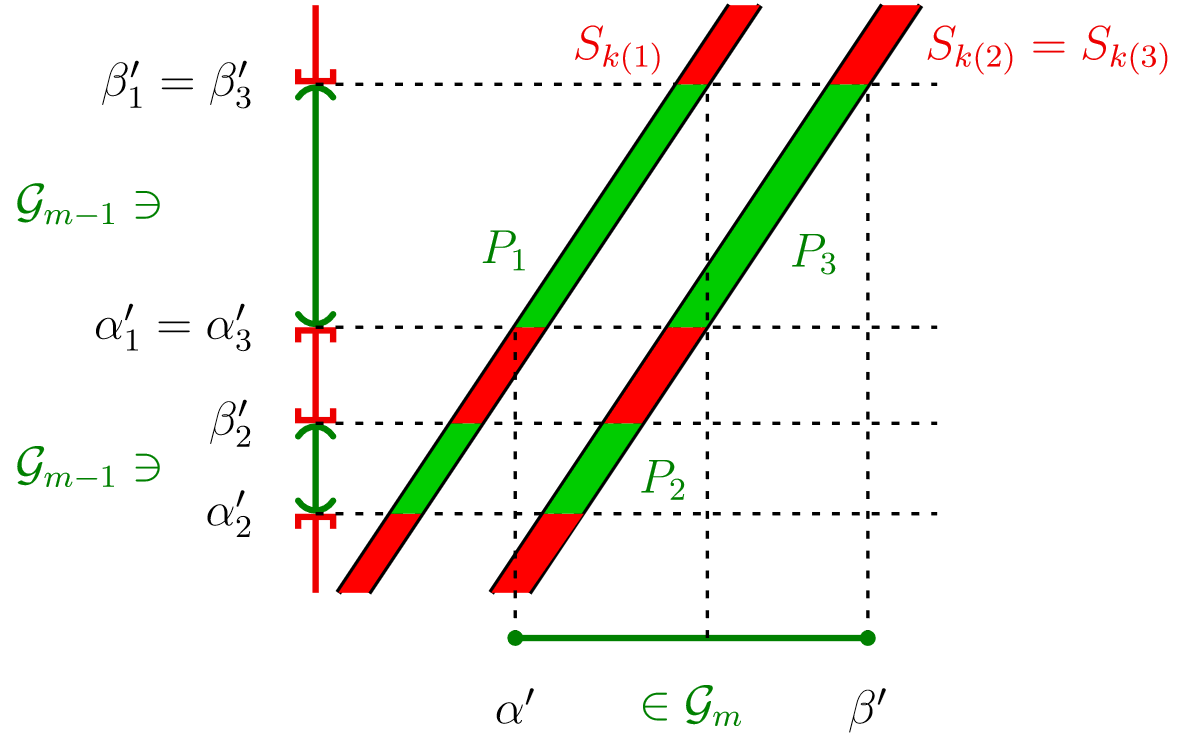}
    \caption{The definition of the parallelograms $P_j$. }
    \label{y95}
    \end{figure}
\end{fact}

\begin{proof}
This is immediate from the definitions
    \eqref{z67} and \eqref{z66}.
\end{proof}

\begin{claim}\label{z20}
  Let $1\leq m\leq N_0$ and $(\alpha ',\beta ')\in \mathcal{G}_m$. With the notation of Fact \ref{z25}, for all $j\in[\ell ]$,
  \begin{equation}
  \label{z19}
  \bigcup\limits_{x\in \pi _1\left( P_j\right)\cap T(\varepsilon )}
  \left( \ell  _{k(j)}^{1}(x),\ell  _{k(j)}^{2}(x) \right)\cap T(\varepsilon )
  \supset
  (\alpha '_j,\beta '_j)\cap T(\varepsilon ).
  \end{equation}
\end{claim}
\begin{proof}
 Let $z\in  (\alpha '_j,\beta '_j)\cap T(\varepsilon )$.
 Then it follows from \eqref{z21} that
there exists a $w\in L(z,k(j))\cap T(\varepsilon )\subset \pi _1(P_j)\cap T(\varepsilon )$ such that
$z\in (\ell  _{k(j)}^{1 }),\ell  _{k(j)}^{2 })$.
\end{proof}

Before we start the proof of Proposition \ref{z29} we   recall some definitions: We defined $m_I$ in \eqref{z81} and we saw that $\mathrm{supp}(m_I)=\bigcup\limits_{k=1}^{ M} S_k$.
\fm   Moreover,   in the statement of the Main Lemma we defined
\begin{equation}
\label{z09}
m^{\varepsilon }=m_I\cdot \indicator_{T(\varepsilon)\times T(\varepsilon)}
  \text{ and }   m^{\varepsilon}_1:=m^{\varepsilon}.
\end{equation}
 Furthermore, for $n\geq 1$:
\begin{equation}
\label{z11}
m^{\varepsilon}_{n+1}(x,y):= \int\limits_{T(\varepsilon)}  m^{\varepsilon}_{n}(x,z)\cdot m^{\varepsilon}_{1}(z,y)\, \mathrm dz.
\end{equation}
   Similarly, for a $k\geq 1$, we defined
   \begin{equation}
   \label{z14}
   E _{k}^{\varepsilon }(x_0):= \left\{ (x,y): x=x_0 \right\}\cap \mathrm{supp}(m _{k}^{\varepsilon  })
   = \left\{y\in T(\varepsilon ): m _{k}^{\varepsilon  }(x_0,y)>0  \right\}.
   \end{equation}
  In particular, $E _{1}^{\varepsilon  }(z)=\emptyset $ if $z\not\in T(\varepsilon )$ and
  \begin{equation}
  \label{z06}
  E _{1}^{\varepsilon  }(z)=  \left(\bigcup\limits_{k=1}^{M} \left( \ell_{k}^{1}(z),\ell  _{k}^{2}(z) \right)\right)\cap T\quad\text{ if }\quad z\in T(\varepsilon ).
  \end{equation}

\fm\fm \begin{proof}[Proof of Proposition \ref{z29}]
We fix an $0<\varepsilon <\varepsilon _{{\rm\scriptscriptstyle{MAIN}}}$ and we write
  $T:=T(\varepsilon )$, $m_k:=m _{k}^{\varepsilon }$, $E_k:= E_{k }^{ \varepsilon }$ and $\kappa :=\kappa (\varepsilon )$ (defined in Claim \ref{z37}).
We divide the proof into two steps.
First we recall (see  part (a) of Definition \ref{z45}) that the connected components of \cblue{\fm} $T(0)=V_{N_0}$ are $\left\{(\alpha_i,\beta_i)\right\}_{i=1}^{\tau}$. So, $(\alpha_i,\beta_i)\in\mathcal{G}_{N_0}$ for all $i\in[\tau]$.

\textbf{Step 1}
There exists an $N$ such that for all $x\in T$ here exists an $n(x)\leq N$ such that there exists {an} $1\leq i\leq \tau $ with
\begin{equation}
\label{z13}
[\alpha _i+\varepsilon ,\beta _i-\varepsilon ]\subset E_{n(x)}(x).
\end{equation}
\textbf{Step 2}
For every $x\in T$ we have
\begin{equation}
\label{z12}
[\alpha _i+\varepsilon ,\beta _i-\varepsilon ]\subset E_{n(x)}(x) \Longrightarrow E_{n(x)+N_0}(x)=T.
\end{equation}

\begin{proof}[Proof of Step 1]
It follows from \eqref{z11}, \eqref{z14}, \eqref{z09} and \eqref{z06} that
\begin{equation}
\label{z10}
E_{n+1}(x)=\bigcup\limits_{z\in E_n(x)}E_1(z)=\bigcup\limits_{z\in E_n(x)} \left(\bigcup\limits_{k=1}^{  {M}}
  \left( \ell  _{k}^{1}(z),\ell  _{k}^{2}(z) \right) \right)\cap T.
\end{equation}
  {Let $x\in T$.}
Using \eqref{z33} there exists an $i\in[\tau ]$, $k\in[M]$ such that $x\in \pi _1(H_{i,k})$. Then either
\begin{enumerate}
[{\bf (a)}]
  \item $(\ell  _{k}^{1 }(x),\ell  _{k}^{2 }(x))\subset [\alpha _i+\varepsilon ,\beta _i-\varepsilon ]$ or
  \item either $\alpha _i+\varepsilon$ or
$\beta _i-\varepsilon$ is an endpoint of a component of $E_1(x)$ of
  length at least $\kappa $ (c.f. \eqref{z34}).
\end{enumerate}
If (b) holds then $E_1(x)$ has a component of length at least $\kappa$ which has at least one endpoint in $T$.

If (a) holds then we consider
$$
E_2(x)=\bigcup\limits_{z\in E_1(x)}E_1(z)\supset
\bigcup\limits_{z\in (\ell  _{k}^{1 }(x),\ell  _{k}^{2 }(x))
}E_1(z).
$$
If $J:=\bigcup\limits_{z\in (\ell  _{k}^{1 }(x),\ell  _{k}^{2 }(x))
}E_1(z)$ is still contained in a component of $T$ then $J$ is an interval of length at least $(1/a)\kappa $. We continue this process and we obtain that if we choose $N_1$ such that $\kappa  (1/a)^{N_1}> \widetilde{\beta }-\widetilde{\alpha }$ then for an $n=n(x)<N_1$, the set
$E_n(x)$ has a component $U_1$ with length at least $\kappa $ and with at least one of its endpoints contained in
$\bigcup\limits_{i=1}^{\tau }\left( {\alpha _i+\varepsilon ,\beta_i -\varepsilon}  \right)$.
Without loss of generality, we may assume that
$$
U_1=(\alpha _i+\varepsilon ,w) \text{ with } \fmu w-\alpha _i-\varepsilon>\kappa.
$$
See Figure \ref{y94}.
  {Then there exist $k\in [M]$ and $j\in [\tau ]$ such that
\begin{equation}
\label{a98}
\pi _1\left( P _{(\alpha _j,\beta _j)}^{ k}\right)=(\alpha _i,t),
\end{equation}
where $t\leq \beta _i$.
Now, observe that
if $U_1 \supset \pi _1(P _{(\alpha _j,\beta _j)}^{ k})$ then
$E_n(x)\supset \pi _1(P _{(\alpha _j,\beta _j)}^{ k})$ (since $U_i$) is a component of $E_n(x)$. So, by \eqref{z10} and Claim \eqref{z20}
$E_{n+1}(x)\supset [\alpha _j+\varepsilon ,\beta _j-\varepsilon ]     $.}
That is, in this case we are ready. So, from now on we assume that for the  $j$ and $k$ that appear in \eqref{z04} we have
\begin{equation}
\label{z04}
U_1\subset \pi _1\left(
 P _{[\alpha _j+\varepsilon ,\beta _j-\varepsilon ]}^{k } \right).
\end{equation}

\begin{figure}[H]
  \centering
  \includegraphics[width=8cm]{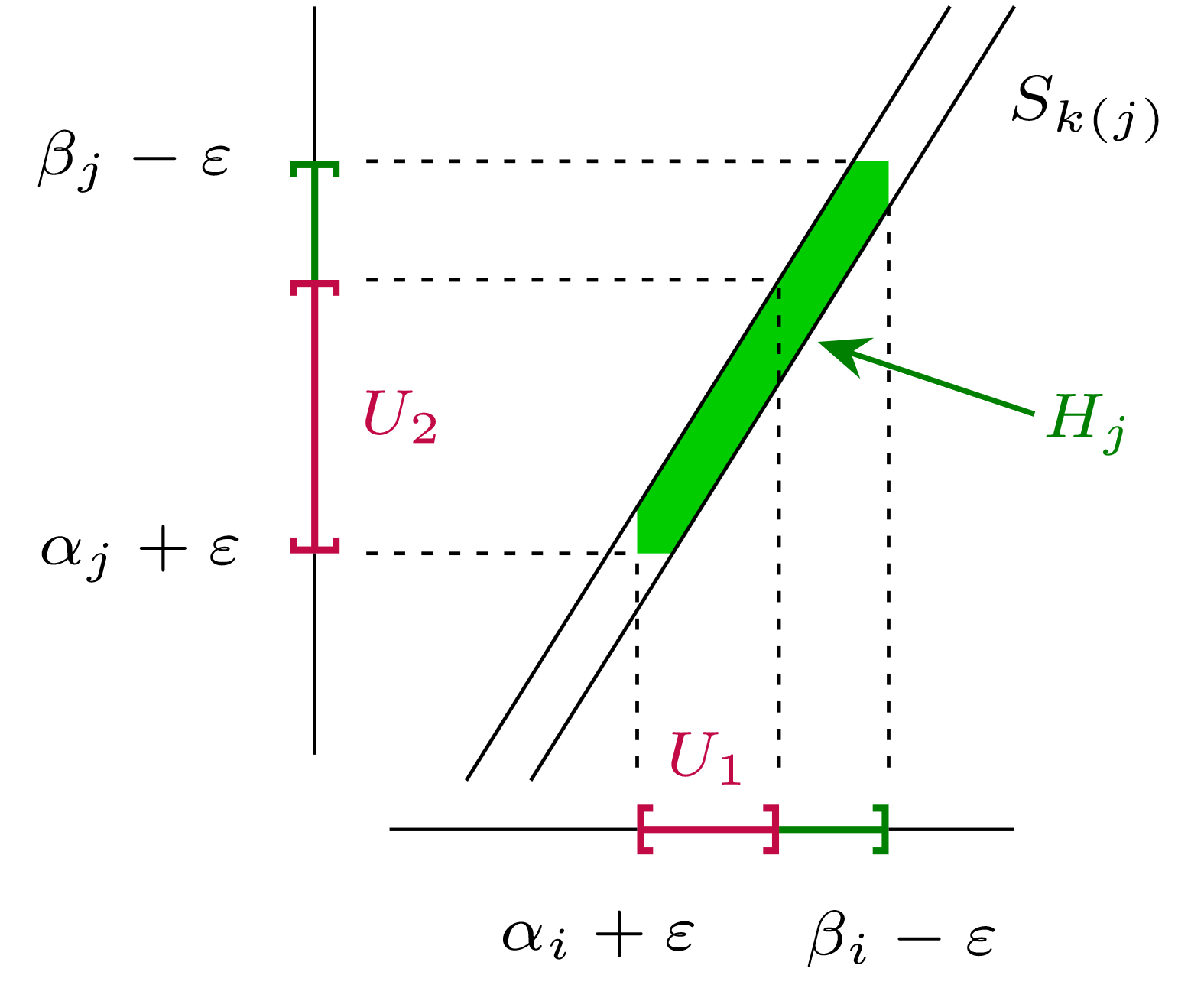}
  \caption{Definition of $U_1$ and $U_2$. }
    \label{y94}
\end{figure}

Now we define
$$
  {U_2:=\bigcup\limits_{z\in U_1}\left(\ell  _{k}^{1 }(z),\ell  _{k}^{2 }(z)\right)\cap T.}
$$
By definition the left endpoint of $U_2$ is $\alpha _j+\varepsilon $.
If $U_2\supset [\alpha _j+\varepsilon ,\beta _j-\varepsilon ]$ that is the right endpoint of $U_2$ is greater than $\beta _j-\varepsilon$
then $E_{n+1}(x)\supset [\alpha _j+\varepsilon ,\beta _j-\varepsilon ]$. Otherwise,
if $U_2\subset [\alpha _j+\varepsilon ,\beta _j-\varepsilon ]$ that is the right endpoint of $U_2$ is not less than $\beta _j-\varepsilon $
then $|U_2|\subset T(\varepsilon )$ and $|U_2|>(1/a)|U_1|$. So, we repeat the same for $U_2$ instead of $U_1$, and we get $U_3$ and so on until after uniformly bounded (not more than $N_2$) steps  $U_k$ contains a component of $T$. This completes the proof for Step 1 with $N=N_1+N_2$.
\end{proof}

\begin{proof}[Proof of Step 2]
  Now we use the notation of Fact \ref{z25}. Let $(\alpha ',\beta ')\in \mathcal{G}_m$ and $(\alpha'_j,\beta'_j)\in \mathcal{G}_{m-1}$ is a child of $(\alpha ',\beta ')$. That is there exists a $k(j)\in[M]$ such that for
  $P_j=P _{(\alpha'_j,\beta'_j) }^{k(j) }$ we have
  $\pi _1(P_j)\subset (\alpha ',\beta ')$. We claim that
\begin{equation}
\label{z03}
(\alpha ',\beta ')\cap T(\varepsilon )\subset E_n(x)
\Longrightarrow
(\alpha '_j,\beta '_j)\cap T(\varepsilon )\subset E_{n+1}(x).
\end{equation}
Namely, assume that $(\alpha ',\beta ')\cap T(\varepsilon )\subset E_n(x)$. Then
\begin{eqnarray}
  E_{n+1}(x) &=& \bigcup\limits_{y\in E_{n}(x)}
  \left(
    \bigcup\limits_{k=1}^{M}
    \left(
      \ell  _{k}^{1}(y),\ell  _{k}^{1}(y)
     \right)
   \right)\cap T
  \\
   &\supset&
   \bigcup\limits_{y\in \pi _1(P_j)\cap T}
   \left(\big(
      \ell  _{k(j)}^{1}(y),\ell  _{k(j)}^{1}(y)\big)\cap T
     \right)
     \supset
     (\alpha '_j,\beta '_j)\cap T(\varepsilon ),
  \end{eqnarray}
where in the last step we used Claim \ref{z20}.
In this way we have proved that \eqref{z03} holds.

Now we fix an $x\in T$ and let $n=n(x)$ be defined as in the proof of Step 1.
We start with $(\alpha _i,\beta _i)$  obtained in the first step.
That is $[\alpha_i+\varepsilon,\beta_i-\varepsilon] \subset E_n(x)$. By Fact \ref{z26} this implies that $(\alpha_i,\beta_i)\cap T \subset E_n(x)$.
We apply \eqref{z03} $N_0$ times. The level $N_0$-th child of $(\alpha _i,\beta _i)\in\mathcal{G}_{N_0}$ is the only element of $\mathcal{G}_{N_0}$,
which is $(\widetilde{\alpha} ,\widetilde{\beta })$. So, we get that
$E_{n+N_0}(x)\supset
(\widetilde{\alpha} ,\widetilde{\beta })\cap T=T$. On the other hand, it is immediate from the definition that $E_{n+N_0}(x) \subset T$.
  \end{proof}
The proof of Proposition \ref{z29} is immediate if we put together what we obtained in Steps 1, 2.
\end{proof}

\begin{proof}[Proof of the Main Lemma]
(1) follows from the definition of $\varepsilon _{{\rm\scriptscriptstyle{MAIN}}}$.

(2) It is easy to see that $m_n(x,y)$ is continuous on $T(\varepsilon )\times T(\varepsilon )$ for $n\geq 2$. We apply Proposition \ref{z29}
to obtain that for an $n\geq2$,  $m_n(x,y)$ is positive for all $(x,y)\in T(\varepsilon )\times T(\varepsilon )$. The uniform positivity follows from the fact, mentioned above, that $m_{n}(x,y)$ is continuous on the compact set $T(\varepsilon )\times T(\varepsilon )$ for $n\geq 2$.

(3) This was proved in Lemma \ref{z39}.

(4) This follows easily from the fact  that the right eigenfunction $f$ is also an eigenfunction of  $(F^{\varepsilon })^2=F^{\varepsilon }\circ F^{\varepsilon }$ whose kernel $m_2$ is continuous.

(5) This was proved in Fact \ref{y70}.

\end{proof}


\section{Appendix: Proof of Proposition \ref{y89}}

 The proof of Proposition \ref{y89} is a simple combination of ideas of \cite[Lemma 2.8]{farkas2019dimension} and \cite[Proposition 6]{Peres-Shmerkin}.
\begin{proof}[Proof of Proposition \ref{y89}]
 Given is  the RIFS $ \mathcal{F}=\left\{f_i(x):f_i(x)=r_ix+D_i\right\}_{i=1}^{L}$ as in Definition
\ref{def:RIFS}, and we write $s:=s(\mathcal{F})$ for the similarity dimension (the solution of the equation \eqref{eq:S_dim}).

\fm The first step is to replace $\mathcal{F}$ by the  RIFS $\mathcal{\widetilde{F}}$ defined in Proposition \ref{prop:Positive}:
\begin{equation}
\label{y36}
\mathcal{\widetilde{F}}=  \left\{ \widetilde{f}_i(x)=\widetilde{r}_ix+\widetilde{D}_i \right\}_{i=1}^{\widetilde{L}},
\end{equation}
\fm which has the convenient property that all contraction ratios $\widetilde{r}_i>0$.\; Further,
let \fm $\widetilde{s}:=s(\widetilde{\mathcal{F}})\ge s-\varepsilon/2$ (\fm where we replaced $\varepsilon$ by $\varepsilon/2$ in Proposition \ref{prop:Positive}).

Let $p_i:=\widetilde{r}_{i}^{ s}$.\fm
For $k_1,\dots,k_m\in\mathbb{N}$ and $k:=k_1+\cdots+k_m$ we introduce
\begin{equation}
\label{y44}
N(k_1,\dots,k_m):=\# \left\{(i_1,\dots  ,i_k): \#\left\{ j\in[k]:i_j=\ell  \right\}=k_\ell,\quad \forall \ell \in[L] \right\}.
\end{equation}

\begin{lemma}[Farkas and Peres-Shmerkin]\label{y43}
  There exists a $C>0$ such that for all $k\in\mathbb{N}$ there \fm exist \fm $k_1,\dots  ,k_m\in\mathbb{N}$,  such that $\sum _{i=1}^{m}k_i=k$ and
\begin{equation}
\label{y42}
N(k_1,\dots  k_{m})\geq C k^{-m/2}p _{1}^{-k_1 }\cdots p _{m}^{-k_{m} }.
\end{equation}
\end{lemma}

Let $C>0$ as in Lemma \ref{y43}. For a large $k$, which will be \fm conveniently chosen at the end of the proof,
we can choose $k_1,\dots  ,k_{\widetilde{L}}\in\mathbb{N}$ according to \fm Lemma \ref{y43} such that we have the following
\begin{equation}
\label{y42}
N(k_1,\dots  k_{\widetilde{L}})\geq C k^{-\widetilde{L}/2}p _{1}^{-k_1 }\cdots p _{\widetilde{L}}^{-k_{\widetilde{L}}}.
\end{equation}
$$
{\rm Let\quad} J_0:=\left\{(i_1,\dots  ,i_k)\in[\widetilde{L}]^k: \#\left\{ j\in[k]:i_j=\ell  \right\}=k_\ell, \quad \forall \ell \in[\widetilde{L}]\right\}.
$$
\fm Note that by definition $\# J_0=N(k_1,\dots  k_{\widetilde{L}})$.
Put $\rho _k:=\prod _{\ell =1}^{\widetilde{L}}\fm\widetilde{r} _{\ell }^{k_{\ell }}$. Then
\begin{equation}
\label{y37}
\mathbf{i}\in J_0 \Longrightarrow r_{\mathbf{i}}=\rho_k\,  \text{ and }\quad \# J_0\geq C\cdot k^{-\widetilde{L}/2}\rho_{k}^{\fm-\widetilde{s}}.
\end{equation}
Let $\mathcal{F}_k:=\left\{f_{\mathbf{i}}:\mathbf{i}\in J_0 \right\}$.
Then the similarity dimension  $s_k:=s(\mathcal{F}_k)$ is the solution of\: $\# J_0\cdot  \rho _{k}^{s_k }=1$.
That is, $\log N(k_1,\dots,k_{\widetilde{L}})+s_k\log \rho _k=0$. Hence,
$$
s_k=\frac{\log N(k_1,\dots,k_{\widetilde{L}})}{-\log \rho_k}\geq\frac{\log C -\widetilde{L}\log \sqrt{k}-\widetilde{s}\log \rho_k}{-\log\rho _k}
     = \frac{\log C -\widetilde{L}\log \sqrt{k}}{-\log\rho _k}+\widetilde{s}.
$$
Using that\fm\; $-\log \rho_k=-\sum _{\ell =1}^{\widetilde{L}} k_{\ell} \log (r_{\ell })\,\ge\,
   -\log(r_{\rm max})\sum _{\ell=1}^{\widetilde{L}}k_{\ell}=-\log(r_{\rm max})\cdot k$ \\
we obtain
$$
\fm s_k\ge\widetilde{s}-\frac{\widetilde{L}\log\sqrt{k}-\log C}{-\log(r_{\rm max})\cdot k}>\widetilde{s}-\frac{\varepsilon }{2}>s-\varepsilon,
$$
if $k$ is large enough.\end{proof}

\vspace*{-1.2cm}

\bibliographystyle{abbrv}

\end{document}